\documentclass[11pt]{article}
\usepackage[utf8]{inputenc}
\usepackage{amsmath}
\usepackage{amssymb}
\usepackage{amsthm}
\usepackage{graphicx}
\usepackage{mathtools}
\usepackage{graphics}
\usepackage{tikz}
\usepackage{authblk}
\usepackage{verbatim}
\usetikzlibrary{graphs,graphs.standard, math, trees}

\providecommand{\keywords}[1]
{
  \small	
  \textbf{\textit{Keywords---}} #1
}

\newtheorem{remark}{Remark}
\newtheorem{theorem}{Theorem}[section]
\newtheorem{definition}{Definition}
\newtheorem{corollary}{Corollary}[theorem]

\title{Characterizing Spectral Properties of Bridge Graphs}
\author{Yixin Li}
\affil{Beijing 101 Middle School}
\date{August 2021}
\begin{document}
\maketitle
\begin{abstract}
The Bridge graph is a special type of graph which are constructed by connecting identical connected graphs with path graphs. We discuss different types of bridge graphs $B_{n\times l}^{m\times k}$ in this paper. In particular, we discuss the following: complete-type bridge graphs, star-type bridge graphs, and full binary tree bridge graphs. We also bound the second eigenvalues of the graph Laplacian of these graphs using methods from Spectral Graph Theory. In general, we prove that for general bridge graphs, $B_{n\times l}^2$, the second eigenvalue of the graph Laplacian should be between $0$ and $2$, inclusive. In the end, we talk about future work on infinite bridge graphs. We created definitions and found the related theorems to support our future work about infinite bridge graphs.
\end{abstract}

\newpage
\tableofcontents
\newpage
\section{Introduction}
Spectral graph theory is the process of characterizing graphs by means of the eigenvalues and eigenvectors of the graph Laplacian. It connects graphs to matrices, and allows us to understand properties of graph using more analytic means. Recently, spectral graph theory has found application in machine learning and deep learning. In particular, there are many clustering algorithms based on spectral methods, like spectral clustering, that are more effective than tradition clustering methods like K-means. Many theoretical properties of these algorithms rely on bounding eigenvalues of the graph Laplacian. \newline

Research has already been done in extracting bounds of eigenvalues of special graphs such as complete graphs, path graphs, the binary tree, and so on. In this paper, we will be focusing on some special types of graphs which are constructed by connecting some identical connected graphs by a path or multiple edges, which we'll call Bridge graphs. Bridge graphs are constructed by using path graphs, $P_m$ with $n\geq 2$, and putting some identical graphs on each end of the path. \newline

We will then bound the second eigenvalues of the graph Laplacians of the graphs we discussed above. The second eigenvalues are the most important because the first eigenvalue of the graph Laplacian is always $0$. We'll use test vectors and Loewner partial ordering to approach this. At the end of this paper, we'll discuss  constructing infinite bridge graphs, which are constructed by connecting a countably infinite number of identical connected graphs using path graphs. We'll also discuss the general idea for bounding the spectrum of the generalized Laplacian operator.

\section{Basic Definitions}
The following definitions are from Dan Spielman\cite{b1}.Assume we have a graph $G$ with vertex set $V$ and edge set $E$. Assume that the number of vertices is $|V|=n$. We can label vertices to be $\{1,2,3,\dots,n-1,n\}$. 
\begin{definition}The adjacency matrix $\bold{M}$ of a weighted graph $G=(V,E, w)$ is defined as the matrix with the following entries
$$
\bold{M}(a,b) = 
\begin{dcases}
w_{a,b} \quad (a,b)\in E\\
0 \quad (a,b)\notin E
\end{dcases}
$$
When the graph is unweighted, $w(a) = 1$ for all $(a,b) \in E$.
\end{definition}
\begin{definition}
The degree of a vertex $a$ is the number of edges attached to it. For a weighted graph, the degree $d(a)$ of the vertex $a$ is the sum of the weights of the edges attached to it. 
\end{definition}

\begin{definition}
The degree matrix $\bold{D}$ of a graph $G=(V,E)$ is a diagonal matrix whose entries are given by
$$
\bold{D}(a,b) = 
\begin{dcases}
d(a)\quad &a=b\\
0 \quad &a\neq b
\end{dcases}
$$
\end{definition}

\begin{definition}
The graph laplacian $\bold{L}$ of a graph $G$ is defined to be
$$
\bold{L} =\bold{M}- \bold{D}.
$$
\end{definition}
\begin{definition}(Loewner partial order)
Let $G_1$ and $G_2$ be graphs each with $n$ vertices. Then for the graph Laplacians of $G_1$ and $G_2$, $L_{G_1}$ and $L_{G_2}$, we write $L_{G_1}  \succcurlyeq L_{G_2}$ if and only if $\bold{v}^TA\bold{v}\geq \bold{v}^TB\bold{v}$ for all vectors $\bold{v} \in \mathbb{R}^n$. The relation $\succcurlyeq$ above is called Loewner partial order. In this case, the graphs $G_1$ and $G_2$ also have relation $G_1 \succcurlyeq G_2$.
\end{definition}
\section{Basic Theorems}
The following theorems are from Dan Spielman\cite{b1}
\begin{theorem}
Assume we have a weighted graph $G=(V,E)$, for every edge $e = (a,b)$, let the weight be $w_{a,b}$. For a function $\bold{x}: V \to \mathbb{R}^n$, the quadratic form of the graph Laplacian is
$$
\bold{x}^T\bold{L}\bold{x}=\sum_{(a,b)\in E} \bold{w}_{a,b}(x(a)-x(b))^2  
$$
\end{theorem}
\begin{theorem}
For a $n\times n$ symmetric matrix $\bold{A}$ with ordered eigenvalues $\lambda_1\le \lambda_2\le \dots\le\lambda_n$ and corresponding eigenvectors $\phi_1,\phi_2,\dots,\phi_n$ we have 
$$\phi_i=\min\limits_{\substack{(\bold{x},\phi_k)=0,\\1\le k \le i-1} }{\frac{\bold{x}^T \bold{L}\bold{x}}{\bold{x}^T\bold{x}}}
$$
with 
$$
\phi_i=\arg\min\limits_{\substack{(\bold{x},\phi_k)=0, \\1\le k \le i-1}} {\frac{\bold{x}^T \bold{L}\bold{x}}{\bold{x}^T\bold{x}}}.
$$
\end{theorem}
\begin{theorem}
For a  graph $G=(V,E)$, with graph Laplacian $L_G$, ordered eigenvalues $\lambda_1\le \lambda_2\le \dots\le\lambda_n$, and corresponding eigenvectors $\phi_1,\phi_2,\dots,\phi_n$, we have $\lambda_1=0$ and $\phi_1=\bold{1}$ where $\bold{1} = (1, \ldots, 1)^{T}$.
\end{theorem}
\begin{theorem}
For a unweigted graph $G=(V,E)$ And $L_G$ is the graph Laplacian with ordered eigenvalues $0=\lambda_1\le \lambda_2\le \dots\le\lambda_n$. Then $G$ is connected if and only if $\lambda_2>0$
\end{theorem}
\begin{theorem}
Suppose $G_1$ and $G_2$ are graphs with  the relation $G_1\succcurlyeq cG_2$. Then $\lambda_k(G_1)\succcurlyeq \lambda_k(G_2)$
\end{theorem}
\begin{theorem}
If $G_1$ is a subgraph of $G_2$ then $G_1\preccurlyeq  cG_2$.
\end{theorem}

\section{$K_n$ Type Bridge Graphs}
Now we will discuss dumbbell-like graphs $D_n^m$, which are formed by joining two complete graphs with $n$ vertices, $K_{n,1}$ and $K_{n,2}$, with a path graph $P_m$. For example, if we connect two $K_8$'s together with $P_3$, we have the following:
\begin{center}
\includegraphics[scale = 0.35]{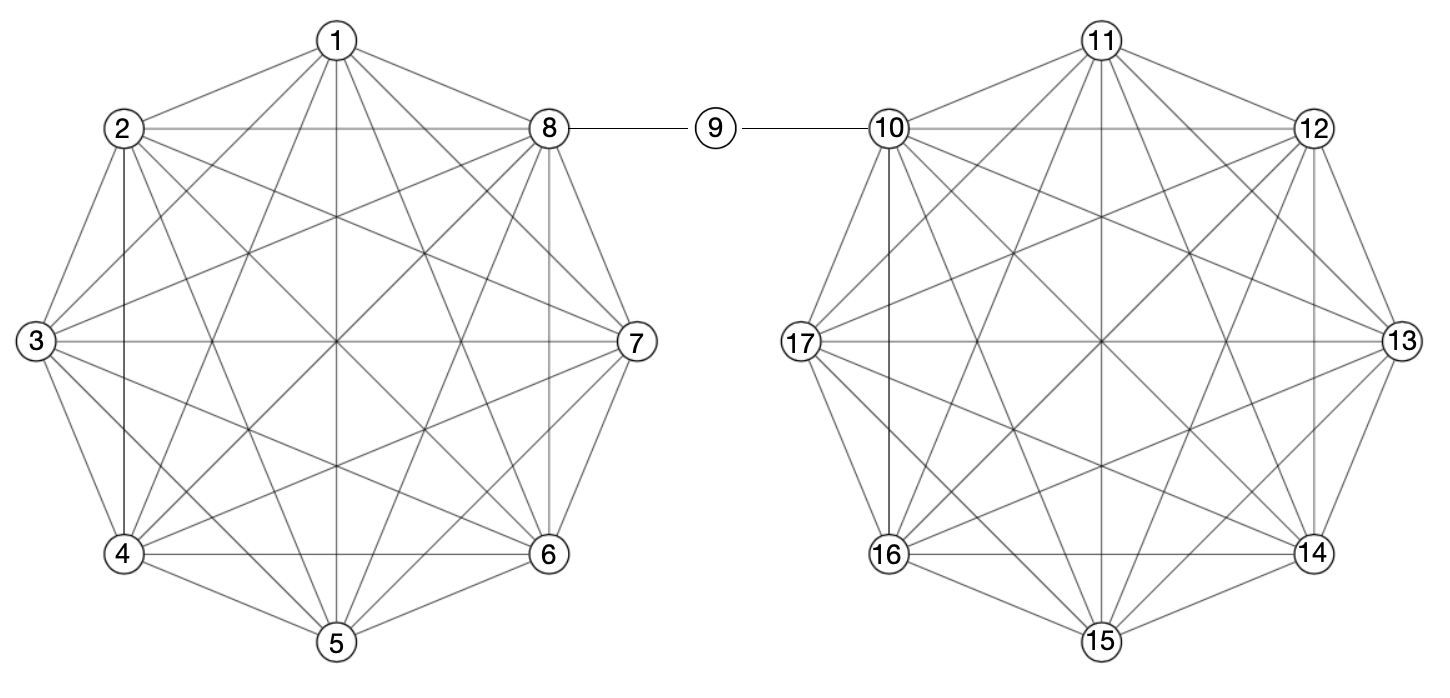}
\end{center}
Notice that this is a simple example of a bridge graph. \\

\begin{corollary}(The Path Inequality)
 A path graph $P_{a,b}$ is a path from $a$ to $b$, and $G_{a,b}$ is a graph with a single edge $(a,b)$, then the following path inequality holds:
 $$
 |P_{a,b}|P_{a,b} \succcurlyeq G_{a,b}.
 $$
\end{corollary}
\begin{theorem}
For the dumbbell-like graph we mentioned above, $D_n^m$, we know that $|V_{D_n^m}|=2n+m-2$ , we have the following bound on the eigenvalues:
$$
\frac{2}{(2n+m-3)(m+1)}\le\lambda_2(D_n^m)\le\frac{12}{6(m-1)(n-1)+m(m-1)}
$$
\begin{proof}
Let $K_{n,1}$ be the first complete graph and $K_{n,2}$ be the second complete graph. We label the vertices of $K_{n,1}$ as $\{1,2,3, \dots ,n-1,n\}$ and suppose that the vertex shared by $K_{n,1}$ and $P_m$ is labeled as $n$. Then the next vertex on $P_m$, which is attached to the vertex $n$ is labeled as $n+1$. Repeat the same process until we label the vertex which is on both $P_m$ and $K_{n,2}$ as $n+m-1$. Finally we label the vertices of $K_{n,2}$ to be $\{n+m-1,n+m,\dots,2n+m-2\}$  \newline

To get the upper bound we construct test vector $\bold{x}$ to be 
$$
\bold{x}(i) = 
\begin{dcases}
m-1\quad &1\le i < n\\
2n+m-1-2i \quad & n\le i <n+m-1\\
1-m \quad &n+m-1\le i \le 2n+m-2
\end{dcases}.
$$
The vertices $n$ and $n+m-1$ are both on one of the complete graphs and the path graphs, so we need to check their value on both graphs to make sure our construction of test vector $\bold{x}$ is consistent. \newline

When $i=n$ we plug into $x(i)=m-1$ gets ${x(i)=m-1}$. Also, if substitute $i = n$ into $x(i)=2n+n-2i$, we get $x(i)=2n+m-1-2n=m-1$. This matches with the other graph. When $i=n+m-1$ we substitute into $x(i)=2n+m-1-2i$, which yields $x(i)=2n+m-1-2(n+m-1)=-m+1$. Also if we substitute $i = n$ into $x(i)=1-m$, we get $x(i)=1-m$, This also matches with the other graph. Hence we have verified the consistency of our test vector.\\

Now we need to calculate the inner product of test vector $\bold{x}$ and the vector $\bold{1}=(1,\dots,1)^T$. We have 
\begin{align*}
(\bold{x},\bold{1}) &= \sum_{i\in V}x(i)\\
&=\sum_{i=1}^{2n+m-2} x(i)\\
&=\sum_{i=1}^{n-1} x(i) + \sum_{i=n}^{n+m-1}x(i)+\sum_{i=n+m}^{2n+m-2} x(i)\\
&=\sum_{i=1}^{n-1} (m-1)+\sum_{i=n}^{n+m-1} (2n+m-1-2i) + \sum_{i=n+m}^{2n+m-2} (1-m).
\end{align*}
We now deal with the three terms separately. For the first term, 
$$\sum_{i=1}^{n-1} (m-1) = (m-1)(n-1).$$
For the second second term, we get
$$\sum_{i=n}^{n+m-1} (2n+m-1-2i) = (2n+m-1)(n+m+1-n-1)-2\cdot\frac{n+m-1+n}{2}.$$
Lastly, 
$$\sum_{i=n+m}^{2n+m-2} (1-m) = (1-m)(n-1).$$
Adding the terms up, we get 
$$(\bold{x},\bold{1}) = 0.$$ \newline

From the calculation above, we know that we can use $\bold{x}$ to get the upper bound of $\lambda_2(D_n^m)$. From Theorem 4 we have 
\begin{align*}
\lambda_2(D_n^m)
&\le \frac{\bold{x}^T\bold{L}\bold{x}}{\bold{x}^T\bold{x}}\\
&\le \frac{\sum_{(a,b)\in E} (x(a)-x(b))^2}{\sum_{i=1}^{2n+m-2}x(i)^2}\\
& = \frac{\sum_{1\le i,j<n} (x(i)-x(j))^2}{\sum_{1\le i<n} x(i)^2+\sum_{n\le i<n+m-1} x(i)^2 +\sum_{n+m-1\le i\le 2n+m-1} x(i)^2}\\
&+ \frac{\sum_{n\le i,j<n+m-1} (x(i)-x(j))^2}{\sum_{1\le i<n} x(i)^2+\sum_{n\le i<n+m-1} x(i)^2 +\sum_{n+m-1\le i\le 2n+m-1} x(i)^2}\\
&+ \frac{\sum_{n+m-1\le i,j\le 2n+m-1} (x(i)-x(j))^2}{\sum_{1\le i<n} x(i)^2+\sum_{n\le i<n+m-1} x(i)^2 +\sum_{n+m-1\le i\le 2n+m-1} x(i)^2}.
\end{align*}
The first and last term of the sum are zero, so this means that 
\begin{align*}
\lambda_2(D_n^m)&=\frac{0+\sum_{i=n}^{n+m-2} (x(i)-x(i+1))^2+0}{(m-1)(n-1)+\frac{m(m-1)(m+1)}{3}+(m-1)(n-1)}\\
&=\frac{12}{6(m-1)(n-1)+m(m-1)}.
\end{align*}
To get the lower bound, we use the Loewner partial order. \newline

For every pair of edge $(a,b)\in E_{D_n^m}$,let the path graph $P_{a,b}$ be a path from $a$ to $b$, and $G_{a,b}$ be a graph with a single edge $(a,b)$, then from the corollary we have $|P_{a,b}|P_{a,b} \succcurlyeq G_{a,b}$. We know that if $1\le a \le n$, then this means $a$ is a vertex of the $K_{n,1}$. Thus, $a$ is connected to vertex $n$. If $n+m-1\le a \le 2n+m-2$ then this means $a$ is a vertex of the $K_{n,2}$ and $a$ is connected to vertex $n+m-1$. If $a$ and $b$ are in the same complete graph then the length of $P_{a,b}$ is $1$; if $a$ and $b$ are in the different complete graphs, then the length of $P_{a,b}$  will be $1+m-1+1=m+1$. If either of $a$ or $b$ are in the path graph, then the length of $P_{a,b}$ shorter than the case when $a$ and $b$ are in different complete graphs. Hence, we conclude that $|P_{a,b}|\le m+1$. \newline

It follows that
$$G_{a,b}\preccurlyeq |P_{a,b}|P_{a,b} \preccurlyeq(m+1)P_{a,b}\preccurlyeq(m+1)D_n^m.
$$
Also, we notice that complete graph $K_{2n+m-2}$ is constructed by connecting all edges together. It has $\binom{2n+m-2}{2}$ single edges.Thus $$K_{2n+m-2}\preccurlyeq \sum_{(a,b)\in E_{K_{2n+m-2}}}G_{a,b}\preccurlyeq \binom{2n+m-2}{2} G_{a,b} \preccurlyeq \binom{2n+m-2}{2}(m+1) D_m^n.
$$
Thus,
$$
2n+m-2=\lambda(K_{2n+m-2}) \le \binom{2n+m-2}{2}(m+1) \lambda(D_m^n).
$$
From the above, we get that $\lambda(D_m^n)\geq \frac{2}{(2n+m-1)(m+1)}$.
\end{proof}

\end{theorem}
We notice that when $m=1$ then $D_n^1$ is a graph constructed by connecting two complete graphs with a single edges. For a bridge graph $D_n^{2\times k}$ with $k\le n$ which is constructed by two identical complete graphs $K_{n,1}$ and $K_{n,2}$ k different edges $e_1,\dots,e_k$ with the edge length of them are all 2, and for every edge 
$e_i=(v_{i,1},v_{i,2})$ where $v_{i,1}$ is in $K_{n,1}$ and $v_{i,2}$ is in $K_{n,2}$. A picture is given below in the case where $n = 8$ with $k=2$ and $e_1=(8,9),e_2=(7,16)$:
\begin{center}
\includegraphics[scale = 0.35]{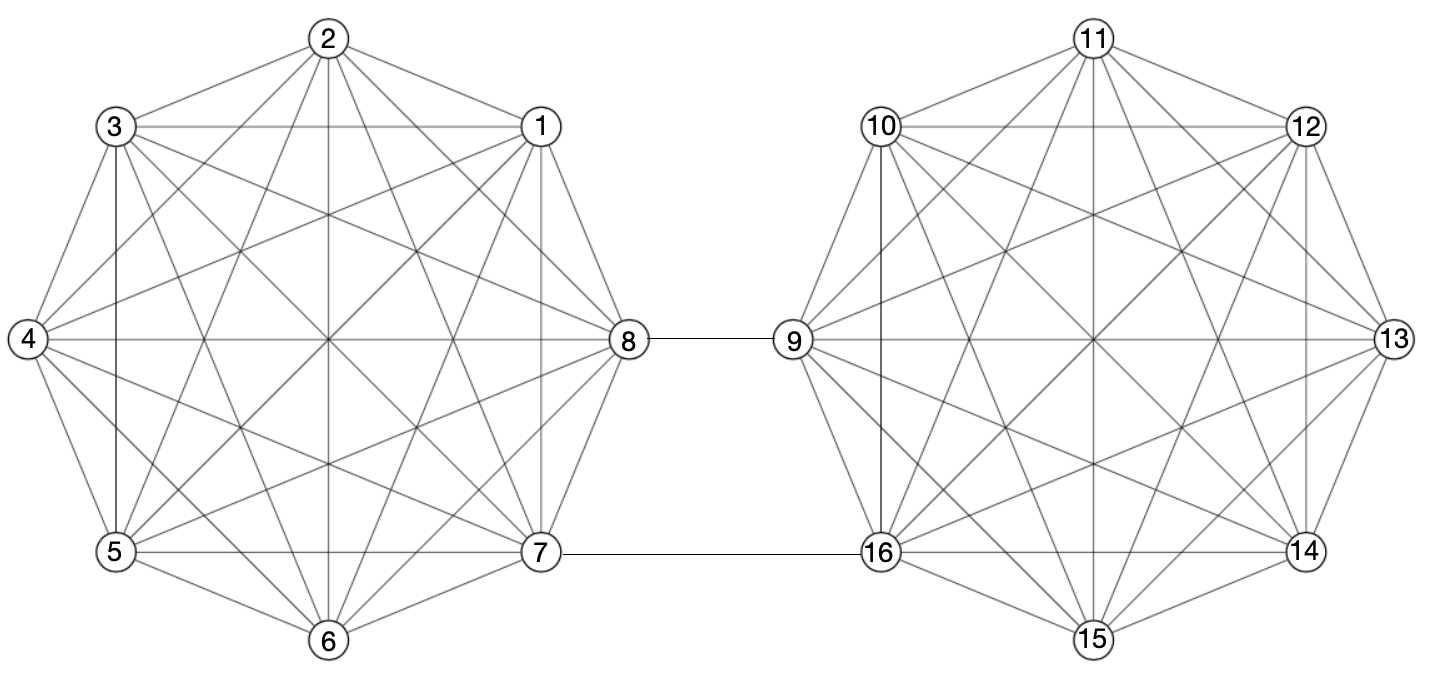}
\end{center}
We can generalize our results from before to the following theorem.

\begin{theorem}
For the graph we mentioned above, $D_n^{2\times K}$, we know that $|V_{D_n^{2\times 2}}|=2n+1$. We also have the following bound on the second eigenvalue of the graph laplacian:
$$
\frac{2}{3(2n-1)}\le\lambda_2(D_n^{2\times k})\le\frac{4}{n}.
$$
\begin{proof}
From $D_n^2$ is a subgraph of $D_n^{2\times k}$ we can get  $D_n^2 \preccurlyeq D_n^{2\times k}$. Thus $\frac{2}{3(2n-1)}=\lambda_2(D_n^2) \le\lambda_2(D_n^{2\times k})$. 
For the other half of the inequality we can label vertices the same way as the graph $D_n^2$. We know vertices $v_{i,1}$ and $v_{i,2}$ are adjacent to each other for $1\le i \le k$ such that $v_{i,i} \in K_{n,1}$ and $v_{i,2} \in K_{n,2}$. Then we have $1\le v_{i,1}\le n$ and $n+1\le v_{i,2}\le 2n$. Also, we can use the same test vector $\bold{x}$ as the graph $D_n^m$ too. Let
$$
\bold{x}(i) = 
\begin{dcases}
1\quad &1\le i \le n\\
-1 \quad &n+1\le i \le 2n
\end{dcases}.
$$
We can easily verify that $(\bold{x},\bold{1})=0$. \newline

Now we can estimate the upper bound of $\lambda_2({D_n^{2\times2}})$:
\begin{align*}
\lambda_2(D_n^{2\times 2})
&\le \frac{\bold{x}^T\bold{L}\bold{x}}{\bold{x}^T\bold{x}}\\
&\le \frac{\sum_{(a,b)\in E_{D_n^{2\times 2}}} (x(a)-x(b))^2}{\sum_{i=1}^{2n}x(i)^2}\\
& = \frac{\sum_{1\le i,j<n} (x(i)-x(j))^2}{\sum_{1\le i\le n} x(i)^2+\sum_{n+1\le i\le 2n} x(i)^2}\\
&+ \frac{\sum_{i=1}^{k}(x(v_{i,1})-x(v_{i,2}))^2}{\sum_{1\le i\le n} x(i)^2+\sum_{n+1\le i\le 2n} x(i)^2}\\
&+ \frac{\sum_{n+2\le i,j\le 2n} (x(i)-x(j))^2}{\sum_{1\le i\le n} x(i)^2+\sum_{n+1\le i\le 2n} x(i)^2}.
\end{align*}
The first and last term of the sum are zero, so this means that 
\begin{align*}
\lambda_2(D_n^m)&\le \frac{0+4k+0}{n+n}\\
&=\frac{4k}{2n}\\
&=\frac{2k}{n}.
\end{align*}
So we have finished bounding $D_n^{2\times k}$.
\end{proof}
\end{theorem}
Consider a general bridge graph $B_n^{2\times k}$ which is constructed by two arbitrary identical graphs $G_{n,1}$ and $G_{n,2}$ with $k$ different edges $e_1,\dots,e_k$ and $k\le n$. We also assume that there are $k$ distinct edges connecting the two graphs. An example is given in the figure:
\begin{center}
\includegraphics[scale = 0.35]{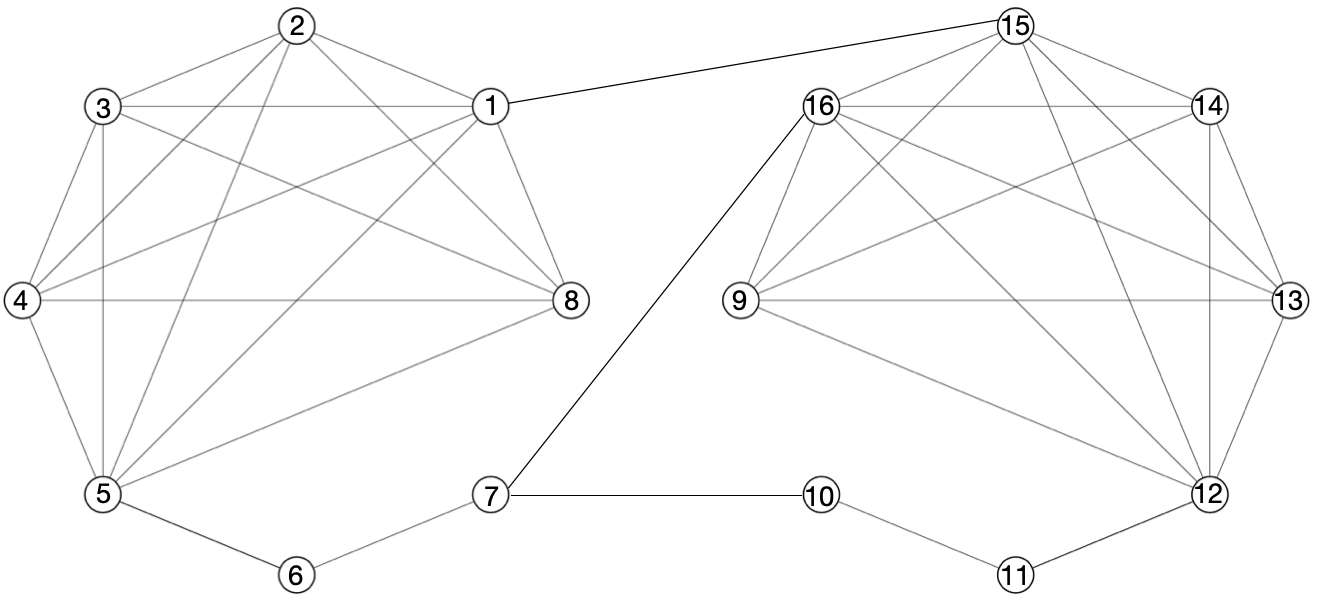}
\end{center}
For every edge 
$e_i=(v_{i,1},v_{i,2})$, where $v_{i,1}$ is in $G_{n,1}$ and $v_{i,2}$ is in $G_{n,2}$, we have the following theorem. 
\begin{theorem}
For a bridge graph $B_n^{2 \times k}$ we have 
$$
0< \lambda_2(B_n^{2\times k}) \le \frac{2k}{n}.
$$
\begin{proof}
The proof of the lower bound follows from the fact that $B_m^{2\times k}$ is a connected graph. Thus the second eigenvalue should be positive . The upper bound is a straight forward consequence of Theorem 4.2. We see that $B_n^{2\times k}$ is a subgraph of $B_n^{2\times k}$; hence $B_n^{2\times k}\preccurlyeq  D_n^{2\times k}$. It follows that $\lambda_2(B_n^{2\times k}) \le \lambda_2(D_n^{2\times k})\le \frac{2k}{n}$.
\end{proof}
\end{theorem}

\section{$S_n$ Type Bridge Graphs}
The star graph, $S_n$, is another graph we will consider. The star graph is special because it is a complete bipartite graph, $K_{1,n-1}$. Now we can construct star-type bridge graphs, $S_n^m$, by connecting two identical star graph $S_{n,1}$ and $S_{n,2}$ with a path graph $P_m$.

\begin{theorem}
For the  star-like graphs $S_n^m$ we mentioned above, we have the following bound on the eigenvalues:
$$
\frac{2}{(2n+m-3)(m+3)} \le \lambda_2(S_n^m) \le \frac{4n+2}{2n+m-4}.
$$
\begin{proof}
Since $S_{n,1}$ is also a bipartite graph ,we can separate it to different set $V_{(1,1)}$ and $V_{(1,2)}$ where $V_{(1,1)}$ only has one vertex which is internal vertex for the tree $S_{n,1}$. And the remaining vertices are all leaves of tree and they are in $V_{(1,2)}$. Notice that no edges has both vertices in the same sets, and every edges that connect vertices in different set is part of the graph. We label the only vertex in $V_{(1,1)}$ as $1$ and remaining as $2,\dots,n$. \newline

We know that there is a vertex $v^1$ in $S_{n,1}$ is also on the graph $P_m$. Then we label the vertex which is attached to $v^1$ but not in graph $S_{n,1}$ as $n+1$, repeat the same process until we label the vertex $n+m-2$.We can also separate it to different set $V_{(2,1)}$ and $V_{(2,2)}$ where $V_{(2,2)}$ only has one vertex which is internal vertex for the tree $S_{n,2}$. And the remaining vertices are all leaves of tree and they are in $V_{(2,2)}$. Also notice that no edges has both vertices in the same sets, and every edges that connect vertices in different set is part of the graph. We label the only vertex in $V_{(1,1)}$ as $n+m-1$ and remaining as $n+m,\dots,2n$. We notice that there is a vertex $v^2$ in $S_{n,2}$ is also on the graph $P_m$. Thus $n+m-1\le v^2 \le 2n$. \newline

Now we can set the test vector. Now we need to discuss different cases based on whether $n$ is odd or even and based on the value of $v^1$ and $v^2$. \newline

\textbf{Case 1:}
$n$ is odd, $v^1=1$ and $v^2=n+m-1$, the figure below is the case when $n=9,m=3$ \newline
\begin{center}
    \includegraphics[scale = 0.20]{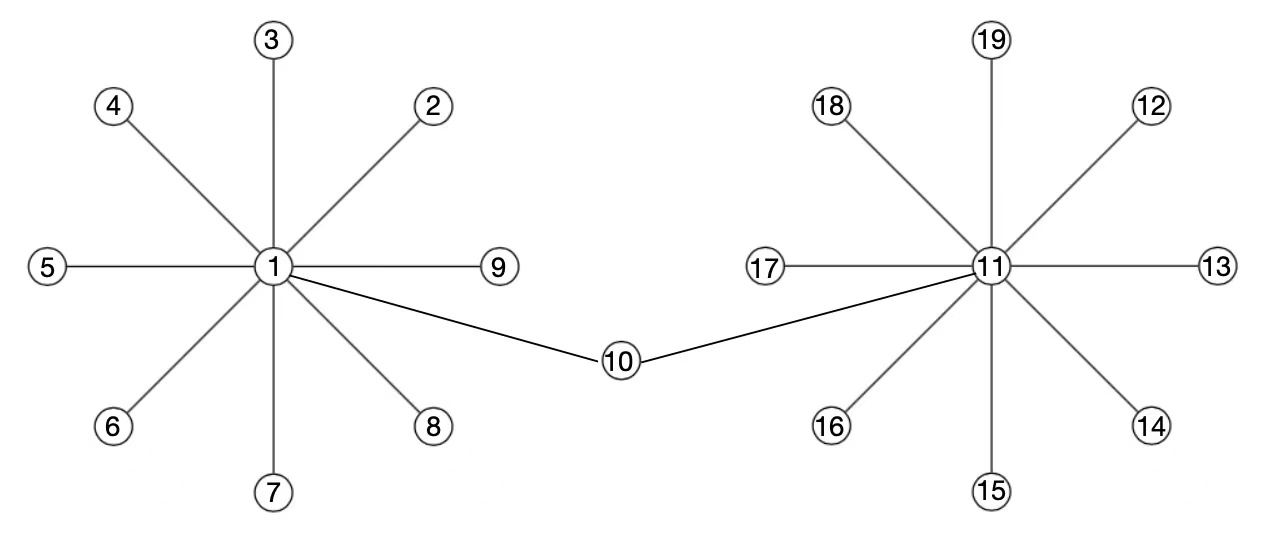}    
\end{center}
We choose the test vector 
$$
\bold{x}(i) = 
\begin{dcases}
1\quad &i=1,n+m-1\\
0\quad &i=n \text{ or } i=2n \text{ or } n+1\le i \le n+m-2\\
1 \quad &2\le i \le \frac{n-1}{2}\\
-1 \quad &\frac{n+1}{2}\le i \le n-1\\
1 \quad &n+m\le i \le \frac{3n+2m-3}{2}\\
-1\quad &\frac{3n+2m-1}{2}\le i \le 2n-1 \\
\end{dcases}.
$$
Notice that 
\begin{align*}
(\bold{x},\bold{1})&=\sum_{i=1}^{2n+m-2} x(i)\\
&=x(1)+x(n+m-1)+\sum_{i=n+1}^{n+m-2}x(i)+\sum_{i=2}^{\frac{n-1}{2}}x(i)+\sum_{\frac{n+1}{2}}^{n-1} x(i)\\
&+\sum_{i=n+m}^{\frac{3n+m-3}{2}}x(i)+\sum_{\frac{3n+m-1}{2}}^{2n} x(i)\\
&=1+1+0+\left(\frac{n-1}{2}-2+1\right) -\left(n-1-\frac{n+1}{2}+1\right)\\
&+\left(\frac{3n+2m-3}{2}-(n+m)+1\right)-\left(2n-1-\frac{3n+2m-1}{2}+1\right)\\
&=0.
\end{align*}
Hence it's possible to use our test vector to get the upper bound.
\begin{align*}
&\lambda_2(T_n^m)\\
&\le \frac{\bold{x}^T\bold{L}\bold{x}}{\bold{x}^T\bold{x}}\\   
&=\frac{\sum_{(a,b)\in E_{T_n^{m}}} (x(a)-x(b))^2}{\sum_{i=1}^{2n+m-2}x(i)^2}\\
& = \frac{\sum_{1\le i,j\le n} (x(i)-x(j))^2}{\sum_{1\le i\le n} x(i)^2+\sum_{n+1\le i\le n+m-2} x(i)^2+\sum_{n+m-1\le i\le 2n+m-2} x(i)^2}\\
&+ \frac{\sum_{n+1\le i,j\le n+m-2} (x(i)-x(j))^2+(x(1)-x(n+1))^2+(x(n+m-2)-x(n+m-1))^2}{\sum_{1\le i\le n} x(i)^2+\sum_{n+1\le i\le n+m-2} x(i)^2+\sum_{n+m-1\le i\le 2n+m-2}x(i)^2}\\
&+ \frac{\sum_{n+m-1\le i,j\le 2n} (x(i)-x(j))^2}{\sum_{1\le i\le n} x(i)^2+\sum_{n+1\le i\le n+m-2} x(i)^2+\sum_{n+m-1\le i\le 2n}x(i)^2}.
\end{align*}
Notice that the first term is equal to 
\begin{align*}
&\frac{\sum_{2\le j\le n} (x(1)-x(j))^2}{\sum_{1\le i\le n} x(i)^2+\sum_{n+1\le i\le n+m-2} x(i)^2+\sum_{n+m-1\le i\le 2n+m-2} x(i)^2}\\ 
=& \frac{\sum_{2\le j\le n} (x(1)-x(j))^2}{n-1+m-2+n-1}\\
=& \frac{\sum_{j=2}^{\frac{n-1}{2}}(x(1)-x(j))^2+\sum_{j= \frac{n+1}{2}}^{j=n-1}(x(1)-x(j))^2+(x(1)-x(n))^2}{2n+m-4}\\
=&\frac{4\frac{n-1}{2}+1}{2n+m-4}\\
=&\frac{2n-3}{2n+m-4}.
\end{align*}
The second term is equal to 
\begin{align*}
\frac{0+1+1}{2n+m-4}=\frac{2}{2n+m-4}.
\end{align*}
From symmetry, the third term and the second term are the same, so the third term is $\frac{2n-3}{2n+m-4}$. Add all terms together, and we get 
$$
\lambda_2(T_n^m)\le \frac{4n-6}{2n+m-4}.
$$\newline

\textbf{Case 2:}
$n$ is odd, $v^1=1$ and $v^2\neq n+m-1$ or $n$ is odd, $v^1\neq 1$ and $v^2= n+m-1$ \newline
\begin{center}
    \includegraphics[scale = 0.4]{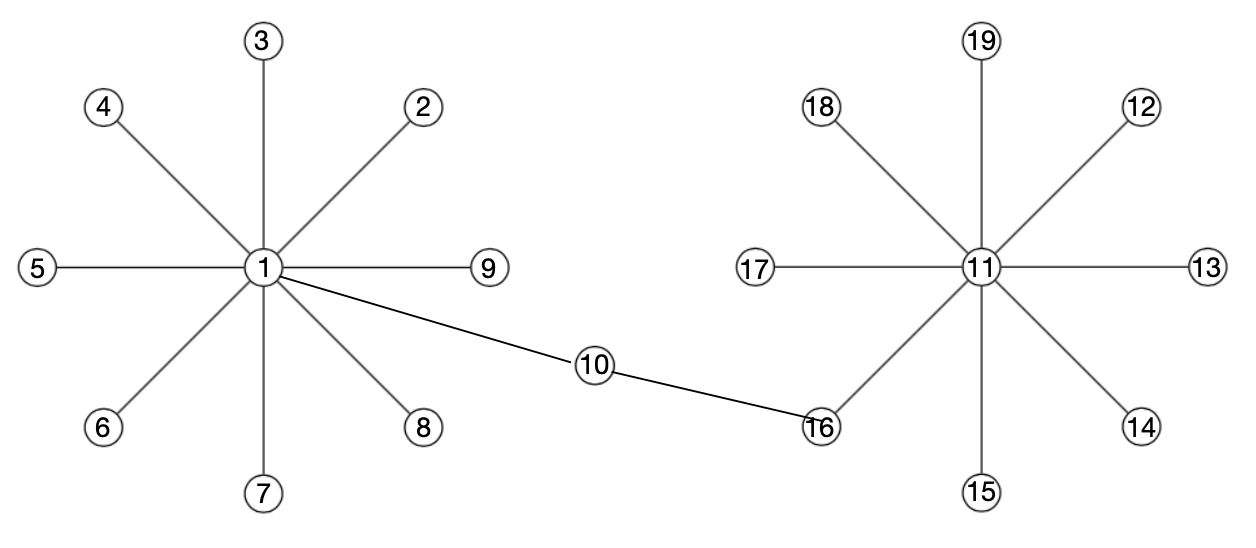}    
\end{center}
We will only discuss when $n$ is odd, $v^1=1$ and $v^2\neq n+m-1$, or $n$ is odd because the other case will get us the same result from symmetry. When we label our vertices, we can make $v^2=n+m$ now. Then we can still use the same test vector as case 1. So the text vector is well defined. The upper bound process will be the same as case 1. Thus we will get the same bound as case 1.\newline

\textbf{Case 3:}
$n$ is odd, $v^1\neq1$ and $v^2\neq n+m-1$ \newline  
\begin{center}
    \includegraphics[scale = 0.4]{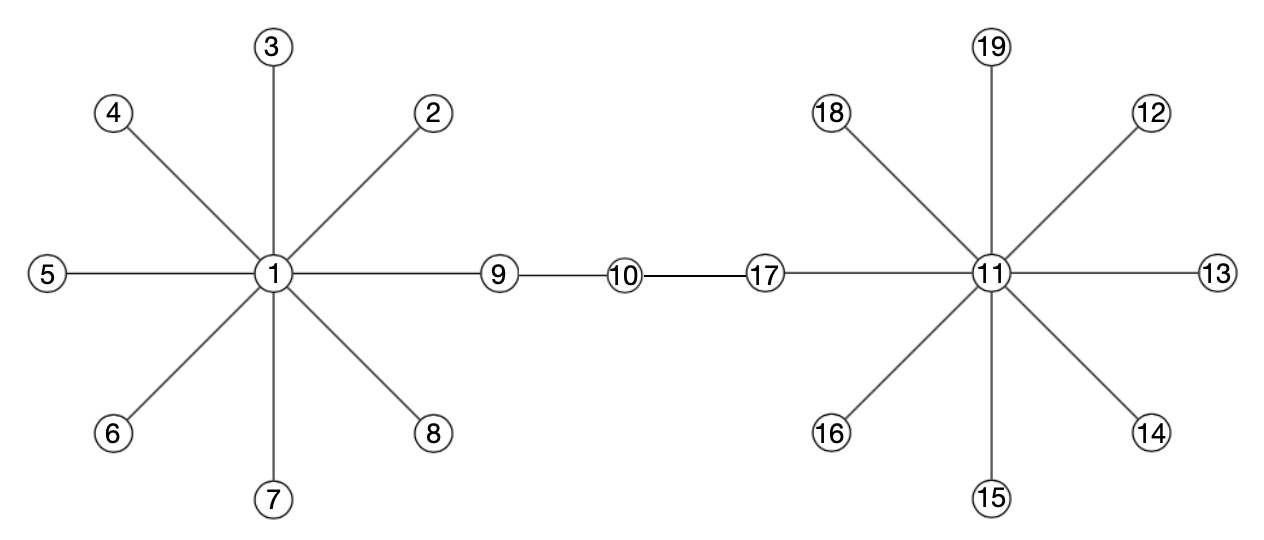}    
\end{center}
When we label the vertices, we can make $v^1=2$ and $v^2=n+m$ now. Then we can still use the same test vector as case 1. So the text vector is well defined. The upper bound process will be the same as case 1 thus we will get the same bound as case 1.\newline

\textbf{Case 4:}
$n$ is even, $v^1=1$ and $v^2=n+m-1$ \newline
\begin{center}
    \includegraphics[scale = 0.4]{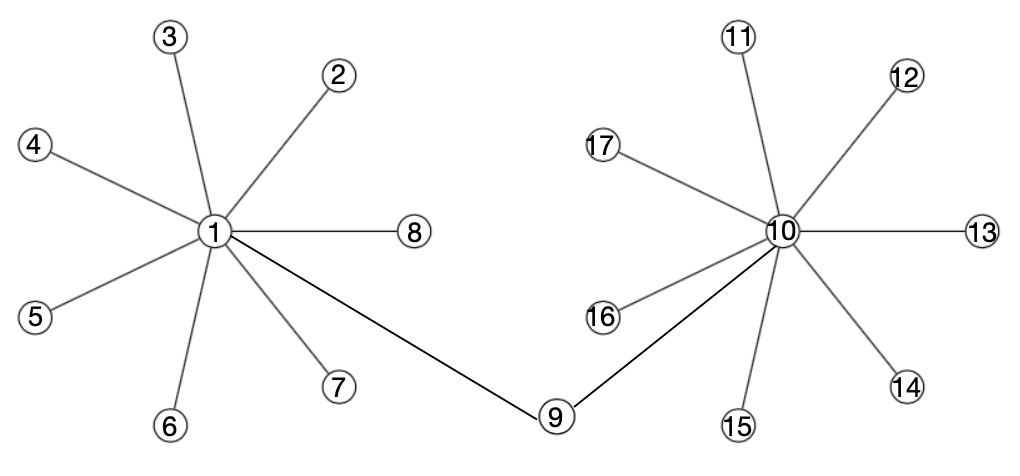}    
\end{center}
We define the test vector as 
$$
\bold{x}(i) = 
\begin{dcases}
1\quad &i=1,n+m-1\\
0\quad & n+1\le i \le n+m-2\\
1 \quad &2\le i \le \frac{n-1}{2}\\
-1 \quad &\frac{n+1}{2}\le i \le n\\
1 \quad &n+m-1\le i \le \frac{3n+2m-3}{2}\\
-1\quad &\frac{3n+2m-1}{2}\le i \le 2n \\
\end{dcases}.
$$
Notice that 
\begin{align*}
(\bold{x},\bold{1})=&=\sum_{i=1}^{2n+m-2} x(i)\\
&=x(1)+x(n+m-1)+\sum_{i=n+1}^{n+m-2}x(i)+\sum_{i=2}^{\frac{n-1}{2}}x(i)\\
&+\sum_{\frac{n+1}{2}}^{n-1} x(i)+\sum_{i=n+m}^{\frac{3n+m-3}{2}}x(i)+\sum_{\frac{3n+m-1}{2}}^{2n} x(i)\\
&=1+1+\left(\frac{n-1}{2}-2+1\right) \\
&-\left(n-1-\frac{n+1}{2}+1\right)+\left(\frac{3n+2m-3}{2}-(n+m)+1\right)\\
&-\left(2n-\frac{3n+2m-1}{2}+1\right)\\
&=0.
\end{align*}
Now we get 
\begin{align*}
&\lambda_2(T_n^m)\\
&\le \frac{\bold{x}^T\bold{L}\bold{x}}{\bold{x}^T\bold{x}}\\   
&=\frac{\sum_{(a,b)\in E_{T_n^{m}}} (x(a)-x(b))^2}{\sum_{i=1}^{2n+m-2}x(i)^2}\\
& = \frac{\sum_{1\le i,j\le n} (x(i)-x(j))^2}{\sum_{1\le i\le n} x(i)^2+\sum_{n+1\le i\le n+m-2} x(i)^2+\sum_{n+m-1\le i\le 2n+m-2} x(i)^2}\\
&+ \frac{\sum_{n+1\le i,j\le n+m-2} (x(i)-x(j))^2+(x(1)-x(n+1))^2+(x(n+m-2)-x(n+m-1))^2}{\sum_{1\le i\le n} x(i)^2+\sum_{n+1\le i\le n+m-2} x(i)^2+\sum_{n+m-1\le i\le 2n+m-2}x(i)^2}.\\
&+ \frac{\sum_{n+m-1\le i,j\le 2n} (x(i)-x(j))^2}{\sum_{1\le i\le n} x(i)^2+\sum_{n+1\le i\le n+m-2} x(i)^2+\sum_{n+m-1\le i\le 2n}x(i)^2}.
\end{align*}
Notice that the first term is equal to 
\begin{align*}
&\frac{\sum_{2\le j\le n} (x(1)-x(j))^2}{\sum_{1\le i\le n} x(i)^2+\sum_{n+1\le i\le n+m-2} x(i)^2+\sum_{n+m-1\le i\le 2n+m-2} x(i)^2}\\ 
=& \frac{\sum_{2\le j\le n} (x(1)-x(j))^2}{n-1+m-2+n-1}\\
=& \frac{\sum_{j=2}^{\frac{n-1}{2}}(x(1)-x(j))^2+\sum_{j= \frac{n+1}{2}}^{j=n-1}(x(1)-x(j))^2}{2n+m-4}\\
=&\frac{4\frac{n}{2}}{2n+m-4}\\
=&\frac{2n}{2n+m-4}.
\end{align*}
From case 1 we know that the second term is $\frac{2}{2n+m-4}$, and from the symmetry, the third term and the second term are the same, so the third term is $\frac{2n}{2n+m-4}$. Adding three terms together gets us
$$
\lambda_2(T_n^m)\le \frac{4n+2}{2n+m-4}.
$$

\textbf{Case 5:}
$n$ is even, $v^1=1$ and $v^2\neq n+m-1$ or $n$ is even, $v^1\neq 1$ and $v^2= n+m-1$ \newline
\begin{center}
    \includegraphics[scale = 0.4]{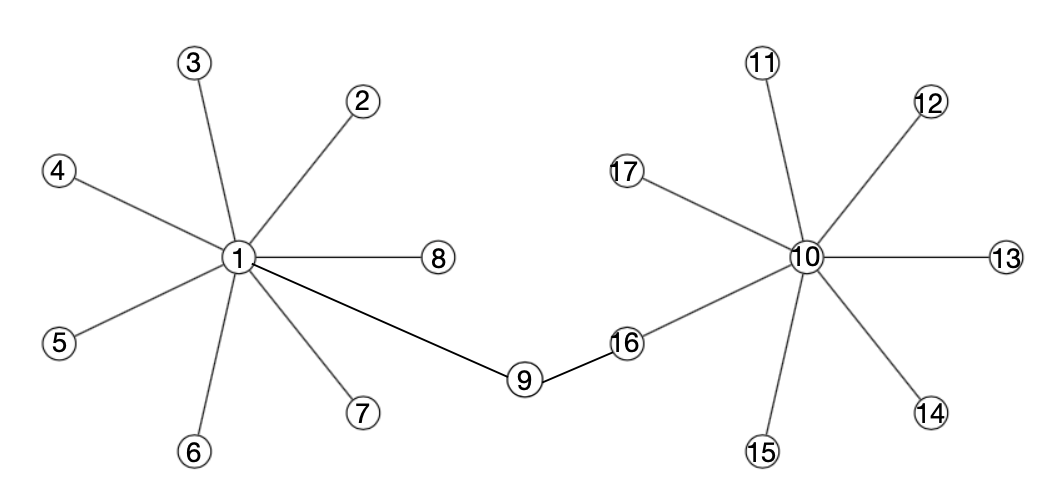}    
\end{center}
We will only discuss when $n$ is odd,$v^1=1$ and $v^2\neq n+m-1$ or $n$ is odd because the other case will get us the same result from symmetry. When we label it we can make $v^2=n+m$ now. Then we can still use the same test vector as case 4. So the text vector is well defined.And the upper bound process will be the same as case 1 thus we will get the same bound as case 1.\newline

\textbf{Case 6:}
$n$ is even, $v^1\neq1$ and $v^2\neq n+m-1$\newline 
\begin{center}
    \includegraphics[scale = 0.4]{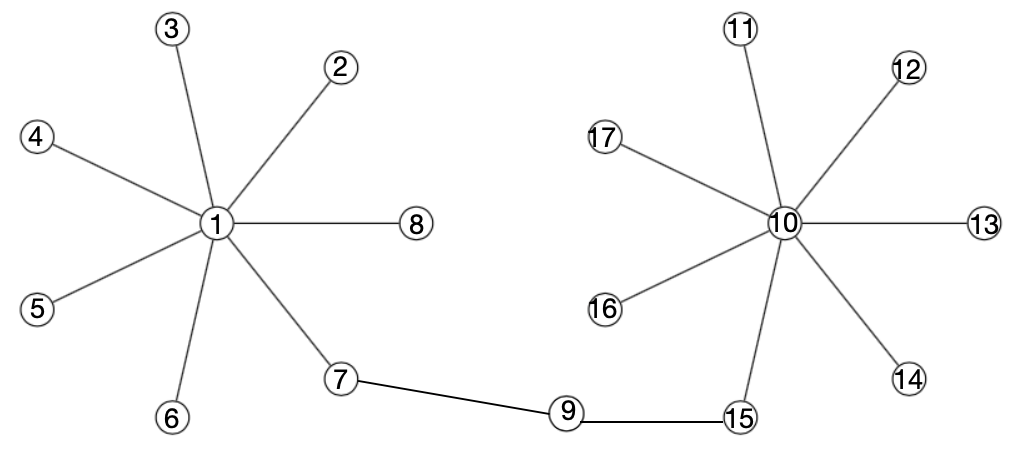}    
\end{center}
When we label it we can make $v^1=2$ and $v^2=n+m$ now. Then we can still use the same test vector as case 4. So the text vector is well defined.And the upper bound process will be the same as case 1 thus we will get the same bound as case 1. \newline

Hence we have finished upper bound since we have exhausted all possible cases. \newline

For the lower bound, we can compare our graphs to complete graphs. For every pair of edge $(a,b)\in E_{S_n^m}$,let the path graph $P_{a,b}$ be a path from $a$ to $b$, and $G_{a,b}$ be a graph with a single edge $(a,b)$. From the lemma we have $|P_{a,b}|P_{a,b} \succcurlyeq G_{a,b}$. We know that if $a$ and $b$ are both in the same  star graph, without loss of generality, we can assume they are both in $S_{n,1}$, Thus, the length from vertex $a$ to $b$ is at most $2$. 

If $a$ and $b$ are in different star graphs, without loss of generality, we suppose $a$ is in $S_{n,1}$ and $b$ is in $S_{n,2}$. Then the length of the path $P_{a,b}$ is at most $2+m-1+2=m+3$.Hence, the length of the path $P_{a,b}$ is at most $2+m-1+2=m+3$. It follows that
\begin{align*}
G_{a,b}\preccurlyeq |P_{a,b}|P_{a,b} & \preccurlyeq(m+3) \\
&\preccurlyeq(m+3)S_n^m.
\end{align*}

Also we know that complete graph $K_{2n+m-2}$ has $\binom{2n+m-2}{2}$ single edges. Thus 
\begin{align*}
K_{2n+m-2}\preccurlyeq \sum_{(a,b)\in E_{K_{2n+m-2}}}G_{a,b}&\preccurlyeq \binom{2n+m-2}{2} G_{a,b} \\
&\preccurlyeq \binom{2n+m-2}{2}(m+3) T_m^n.    
\end{align*}
Hence,
$$
2n+m-2=\lambda_2(K_{2n+m-2}) \le \binom{2n+m-2}{2}(m+3) \lambda_2(S_m^n).
$$
Finally, we arrive at
$$
\lambda_2(S_m^n)\geq \frac{2}{(2n+m-3)(m+3)}.
$$
\end{proof}
\end{theorem}

\begin{section}{$T_n$ Type Bridge Graphs}
Now we will discuss binary tree-like graphs $T_n^m$, which are formed by by joining two full binary trees with $n$ vertices, $T_{n,1}$ and $T_{n,2}$, with a path graph $P_m$. Notice that this is also a simple example of a bridge graph. \\

\begin{theorem}
For the binary tree-like graphs $T_n^m$ we mentioned above we have the following bound on the eigenvalues:
$$
\frac{2}{(2n+m-1)(2 \log_2(n+1)+m-3)}\le \lambda_2{(T_n^m)}\le \frac{5}{2(n-1)}.
$$

\begin{proof}
Now we need to label $T_n^m$. We label $T_{n,1}$ the following way. We label the vertex which is ancestor of $T_{n,1}$ all other vertices as $1$. Then $1$ has two children. We label them as $2$ and $3$,then we label children of $2$ as $4$ and $5$, the children of $3$ as $6$ and $7$ and so on until $n$. \newline

Then we label the path $P_m$. We know one end of $P_m$ is $i$ where $1\le i\le n$, then we label the vertex which is on the path $P_m$ and attached to $0$ as $n+1$ repeat the process until the vertex $n+m-2$, then the next vertex which in $P_m$ and attach to $n+m-2$ is on the graph $T_{n,2}$. \newline

Then we label the graph $T_{n,2}$. We label the vertex which is ancestor of all other vertices of $T_{n,2}$ as $n+m-1$. Then $n+m-1$ has two children. We label them as $n+m$ and $n+m+1$,then we label children of $n+m$ as$n+m+1$ and $n+m+2$, the children of $n+m+1$ as $n+m+3$ and $n+m+4$ and so on until $2n+m-2$. Now we need to break into 3 cases depends on where the ends of $P_m$ locate at. We know that one end is between $1$ and $n$ and the other is between $n+m-1$  and $2n+m-2$. \\

\textbf{Case 1:} One end of $P_m$ is $1$ and the other end of $P_m$ is $n+m-2$. Figure $2$ demonstrates the case of $T_{7}^3$:
\begin{center}
\includegraphics[scale = 0.5]{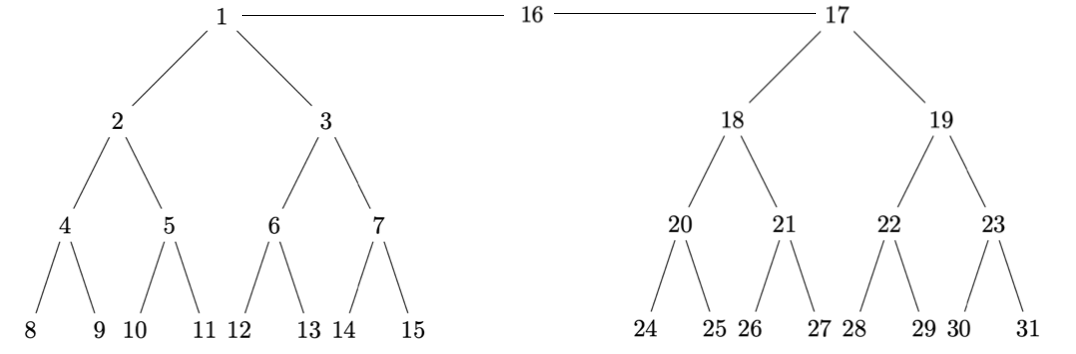}
\end{center}
We can set the test vector to be:
$$
\bold{x}(i) = 
\begin{dcases}
0\quad &i=1,n+m-1\\
0\quad &n+1\le i \le n+m-2\\
1 \quad &i=2,n+m\\
1 \quad &2<i\le n \text{ and } i  \text{ is a descendant of } 2\\
1 \quad &n+m<i\le 2n+m-2 \text{ and } i \text{ is a descendant of } n+m\\
-1\quad & \text{otherwise}
\end{dcases}.
$$
We notice that for elements of $T_{n,1}$ the number of $1$ is $\frac{n-1}{2}$ and the number of $-1$ is $\frac{n-1}{2}$. For elements of $T_{n,1}$ the number of $1$ is $\frac{n-1}{2}$ and the number of $-1$ is $\frac{n-1}{2}$. Hence we have 
\begin{align*}
(\bold{x},\bold{1})&=\sum_{i=1}^{2n+m-2} x(i)\\
&=x(1)+x(n+m-1) + \sum_{i=2}^{n}x(i)+\sum_{i=n+m}^{2n+m-2} x(i)\\
&=0+0+\frac{n-1}{2}-\frac{n-1}{2}+\frac{n-1}{2}-\frac{n-1}{2}\\
&=0.
\end{align*}
We have finished verifying $(\bold{x},\bold{1})=0$. Now we estimate the upper bound of $\lambda_2(T_n^m)$:
\begin{align*}
\lambda_2(T_n^m)&\le \frac{\bold{x}^T\bold{L}\bold{x}}{\bold{x}^T\bold{x}}\\   
&=\frac{\sum_{(a,b)\in E_{T_n^{m}}} (x(a)-x(b))^2}{\sum_{i=1}^{2n+m-2}x(i)^2}\\
& = \frac{\sum_{1\le i,j\le n} (x(i)-x(j))^2}{\sum_{1\le i\le n} x(i)^2+\sum_{n+1\le i\le n+m-2} x(i)^2+\sum_{n+m-1\le i\le 2n+m-2} x(i)^2}\\
&+ \frac{\sum_{n+1\le i,j\le n+m-2} (x(i)-x(j))^2+(x(1)-x(n+m-1))^2}{\sum_{1\le i\le n} x(i)^2+\sum_{n+1\le i\le n+m-2} x(i)^2+\sum_{n+m-1\le i\le 2n+m-2}x(i)^2}.\\
&+ \frac{\sum_{n+m-1\le i,j\le 2n} (x(i)-x(j))^2}{\sum_{1\le i\le n} x(i)^2+\sum_{n+1\le i\le n+m-2} x(i)^2+\sum_{n+m-1\le i\le 2n}x(i)^2}.
\end{align*}
Notice that the first term is 
\begin{align*}
&=\frac{ (x(1)-x(2))^2+(x(2)-x(3))^2}{\sum_{1\le i\le n} x(i)^2+\sum_{n+1\le i\le n+m-2} x(i)^2+\sum_{n+m-1\le i\le 2n+m-1}x(i)^2}\\
&=\frac{2}{n-1+0+n-1}\\
&=\frac{1}{n-1}
\end{align*}
The second term is $0$, and the third term is 
\begin{align*}
&=\frac{ (x(n+m-1)-x(n+m))^2+(x(n+m-1)-x(n+m+1))^2}{\sum_{1\le i\le n} x(i)^2+\sum_{n+1\le i\le n+m-2} x(i)^2+\sum_{n+m-1\le i\le 2n+m-1}x(i)^2}\\
&=\frac{2}{n-1+0+n-1}\\
&=\frac{1}{n-1}.
\end{align*}
Adding all the three terms, we get 
$$
\lambda_2({T_n^m})\le \frac{1}{n-1}+\frac{1}{n-1}=\frac{2}{n-1}.
$$\newline
\textbf{Case 2:} One end of $P_m$ is $1$, and the other end of $P_m$ is $J$ where $n+m\le J \le 2n$ or one end of  $P_m$ is $n+m-1$ and the other end of $P_m$ is $K$ where $1\le K \le n$. Without loss of generality we only discuss when one end of $P_m$ is $1$ and the other end of $P_m$ is $J$ because the other case follows by an identical argument. Figure $3$ demonstrates $T_7^3$ in this case:
\begin{center}
\includegraphics[scale = 0.5]{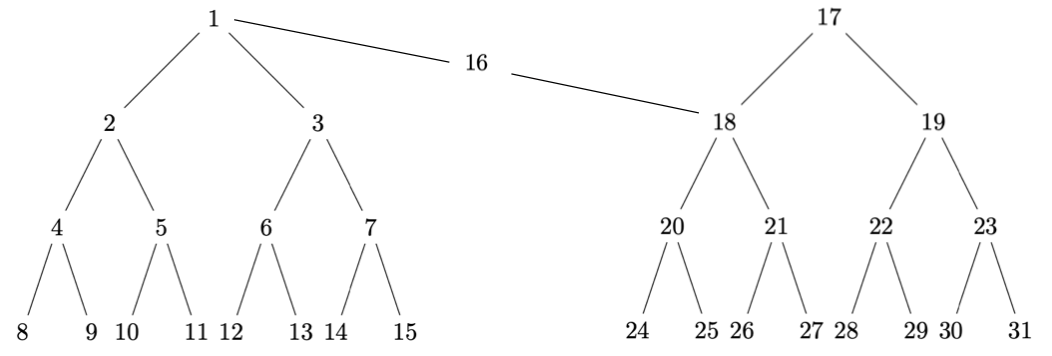}
\end{center}

We can set the test vector to be same as case 1. Since we already know from case 1 that $(\bold{x},\bold{1})=0$,
\begin{align*}
\lambda_2(T_n^m)&\le \frac{\bold{x}^T\bold{L}\bold{x}}{\bold{x}^T\bold{x}}\\   
&=\frac{\sum_{(a,b)\in E_{T_n^{m}}} (x(a)-x(b))^2}{\sum_{i=1}^{2n+m-2}x(i)^2}\\
& = \frac{\sum_{1\le i,j\le n} (x(i)-x(j))^2}{\sum_{1\le i\le n} x(i)^2+\sum_{n+1\le i\le n+m-2} x(i)^2+\sum_{n+m-1\le i\le 2n+m-2} x(i)^2}\\
&+ \frac{\sum_{n+1\le i,j\le n+m-2+(x(1)-x(J))^2} (x(i)-x(j))^2}{\sum_{1\le i\le n} x(i)^2+\sum_{n+1\le i\le n+m-2} x(i)^2+\sum_{n+m-1\le i\le 2n+m-2}x(i)^2}\\
&+ \frac{\sum_{n+m-1\le i,j\le 2n+m-2} (x(i)-x(j))^2}{\sum_{1\le i\le n} x(i)^2+\sum_{n+1\le i\le n+m-2} x(i)^2+\sum_{n+m-1\le i\le 2n+m-2}x(i)^2}.
\end{align*}
Notice that that the second term is 
$$
\frac{0+1}{n-1+n-1}=\frac{1}{2(n-1)}.
$$
We have calculated the first term and the third term in case 1. Hence
$$
\lambda_2({T_n^m})\le \frac{1}{n-1}+\frac{1}{2(n-1)}+\frac{1}{n-1}=\frac{5}{2(n-1)}.
$$\newline
\textbf{Case 3:} One end of $P_m$ is $J$ where $2\le J \le n$ and the other end of $P_m$ is $K$ where $n+m\le K \le 2n$. Figure $4$ demonstrates $T_7^3$ in this case:
\begin{center}
\includegraphics[scale = 0.5]{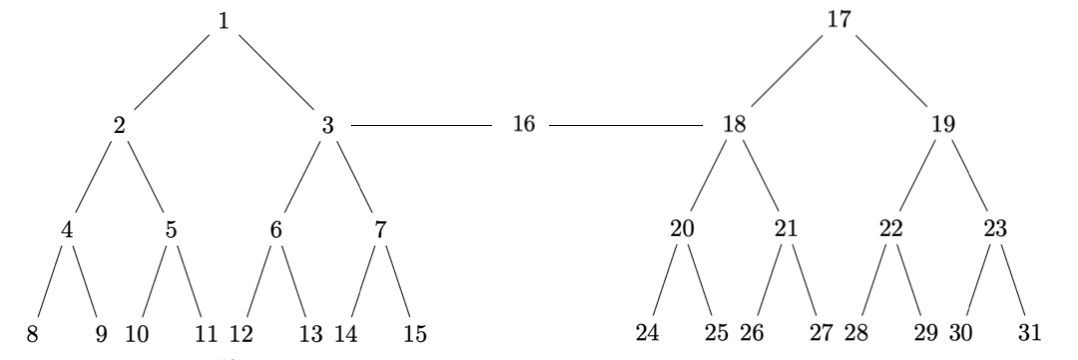}
\end{center}

Similar to before, we can set the test vector to be
$$
\bold{x}(i) = 
\begin{dcases}
0\quad &i=1,n+m-1\\
0\quad &n+1\le i \le n+m-2\\
1 \quad &i=J,K\\
1 \quad &2<i\le n \text{ and } i  \text{ is a descendant of } J  \\
1 \quad &2<i\le n \text{ and } i  \text{ is a ancestor of } J  \\
1 \quad &i<n+m\le n \text{ and } i \text{ is a descendant of } K\\
1 \quad &i<n+m\le n \text{ and } i \text{ is a ancestor of } K\\
-1\quad & \text{otherwise}
\end{dcases}.
$$
We notice that for elements of $T_{n,1}$, the number of $1$'s in $\bold{x}$ is $\frac{n-1}{2}$ and the number of $-1$'s in $\bold{x}$ is $\frac{n-1}{2}$. For elements of $T_{n,1}$ the number of $1$'s in $\bold{x}$ is $\frac{n-1}{2}$ and the number of $-1$'s in $\bold{x}$ is $\frac{n-1}{2}$. Hence we have 
\begin{align*}
(\bold{x},\bold{1})&=\sum_{i=1}^{2n+m-2} x(i)\\
&=x(1)+x(n+m-1) + \sum_{i=2}^{n}x(i)+\sum_{i=n+m}^{2n+m-2} x(i)\\
&=0+0+\frac{n-1}{2}-\frac{n-1}{2}+\frac{n-1}{2}-\frac{n-1}{2}\\
&=0.
\end{align*}
Now we can estimate the upper bound of second eigenvalue
\begin{align*}
\lambda_2(T_n^m)&\le \frac{\bold{x}^T\bold{L}\bold{x}}{\bold{x}^T\bold{x}}\\   
&=\frac{\sum_{(a,b)\in E_{T_n^{m}}} (x(a)-x(b))^2}{\sum_{i=1}^{2n+m-2}x(i)^2}\\
& = \frac{\sum_{1\le i,j\le n} (x(i)-x(j))^2}{\sum_{1\le i\le n} x(i)^2+\sum_{n+1\le i\le n+m-2} x(i)^2+\sum_{n+m-1\le i\le 2n+m-2} x(i)^2}\\
&+ \frac{\sum_{n+1\le i,j\le n+m-2} (x(i)-x(j))^2+(x(J)-x(K)^2)}{\sum_{1\le i\le n} x(i)^2+\sum_{n+1\le i\le n+m-2} x(i)^2+\sum_{n+m-1\le i\le 2n+m-2}x(i)^2}.\\
&+ \frac{\sum_{n+m-1\le i,j\le 2n} (x(i)-x(j))^2}{\sum_{1\le i\le n} x(i)^2+\sum_{n+1\le i\le n+m-2} x(i)^2+\sum_{n+m-1\le i\le 2n+m-2}x(i)^2}.
\end{align*}
Notice that the second term is 
$$
\frac{0}{n-1+n-1}=0.
$$
The first term and the second term calculation is basically the same as case 1, and the result is the same too. It follows that
$$
\lambda_2(T_n^m)\le \frac{1}{n-1}+0+\frac{1}{n-1}=\frac{2}{(n-1)},
$$
which gives us our upper bound estimation. \newline 

For the lower bound, we can compare our graphs to complete graphs. For every pair of edge $(a,b)\in E_{T_n^m}$,let the path graph $P_{a,b}$ be a path from $a$ to $b$, and $G_{a,b}$ be a graph with a single edge $(a,b)$, then from the lemma we have $|P_{a,b}|P_{a,b} \succcurlyeq G_{a,b}$. We know that if $a$ and $b$ are both in the same binary tree, without loss of generality we assume they are both in $P_{n,1}$,from the definition of full binary tree we know that the length from the vertex $1$ to the vertices which have no children is $\log_2{(n+1)}-1$, so the length from vertex $a$ to $b$ is at most $2log_2{(n+1)}-2$. 

If $a$ and $b$ are in different binary trees, without loss of generality, we suppose $a$ is in $T_{n,1}$ and $b$ is in $T_{n,2}$ then the length of the path $P_{a,b}$ is the longest when $a$ and $b$ are the vertices which have no descendants. Hence the length of the path $P_{a,b}$ is at most $2\log_2{(n+1)}+m-3$. That is, $|P_{a,b}|\le 2\log_2(n+1)+m-3$.
It follows that
\begin{align*}
G_{a,b}\preccurlyeq |P_{a,b}|P_{a,b} & \preccurlyeq((2\log_2(n+1)+m-3)P_{a,b} \\
&\preccurlyeq(2\log_2(n+1)+m-3)T_n^m.
\end{align*}

Also we know that complete graph $K_{2n+m-2}$ t has $\binom{2n+m-2}{2}$ single edges. Thus 
\begin{align*}
K_{2n+m-2}\preccurlyeq \sum_{(a,b)\in E_{K_{2n+m-2}}}G_{a,b}&\preccurlyeq \binom{2n+m-2}{2} G_{a,b} \\
&\preccurlyeq \binom{2n+m-2}{2}(2log_2(n+1)+m-3) T_m^n.    
\end{align*}
Hence,
$$
2n+m-2=\lambda(K_{2n+m-2}) \le \binom{2n+m-2}{2}(2log_2(n+1)+m-3) \lambda(T_m^n).
$$
From the above, we conclude $\lambda_2(T_m^n)\geq \frac{2}{(2n+m-1)(2log_2(n+1)+m-3}$.
\end{proof}
\end{theorem}
We notice that when $m=2$ the graph $T_n^2$ is connected by a single edge. Now we are doing the same thing as we did for the complete graphs. When graphs $T_{n,1}$ and $T_{n,2}$ are connected by $k$ different single edges $e_1,e_2,\dots e_k$ we get the graph $T_n^{2\times k}$. An example for $k = 3$ is given below:
\begin{center}
\includegraphics[scale = 0.5]{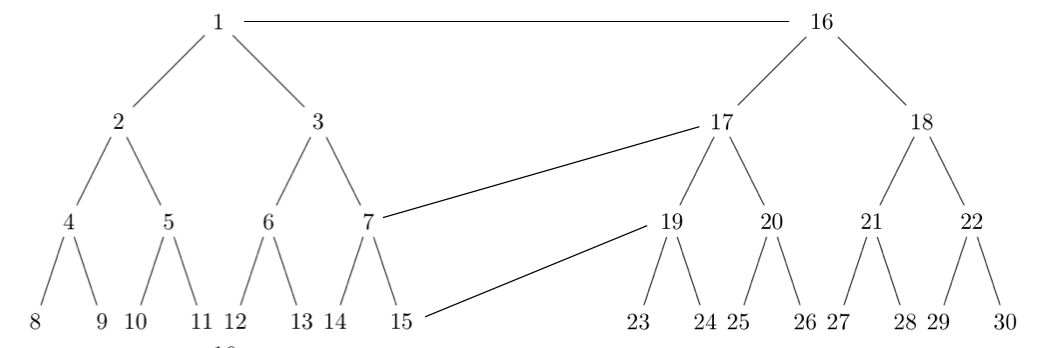}
\end{center}
We now have following theorem.

\begin{theorem}
For the graph $T_n^{2\times k}$ described above, we have the following bound:
$$
\frac{2}{(2n+1)(2\log_2(n-1)+1)}\le \lambda_2(T_n^{2\times k})\le \frac{2m+2}{n-1}+\chi_{k\geq n-1}\frac{-3m+3n-3}{2(n-1)}
$$
where $\chi$ denotes the characteristic function.
\begin{proof}
We know that $T_n^2$ is a subgraph of $T_n^{2\times k}$ so we have $T_n^2\preccurlyeq T_n^{2\times k}$.Hence $$\frac{2}{(2n+1)(2\log_2(n-1)+1)}\le \lambda_2({T_n^2})\le \lambda_2(T_n^{2\times k})
$$
Thus we finished the lower bound. \newline

For the upper bound we still use test vector. Now we notice the graph $T_{n,1}$ contains three different sets of vertices. One set, $V_{1,1}$, only contain the vertex $1$, the second set, $V_{1,2}$, contains vertex $2$ and all of it's descendants. The third set, $V_{1,3}$, contains vertex $3$ and its descendants. The graph $T_{n,2}$ also contains three different vertices sets. One set, $V_{2,1}$, only contains the vertex $n+1$, the second set, $V_{2,2}$, contains vertex $n+2$ and all of it's descendants. The third set, $V_{2,3}$, contains vertex $n+3$ and all of its descendants. \newline

We know that edges $e_1,e_2,\ldots e_k$ contain $k$ vertices in $T_{n,1}$ and $k$ different vertices in $T_{n,2}$. For those $k$ vertices in $T_{n,1}$ we know that they might be in vertex set $V_{1,1}$ or $V_{1,2}$ or $V_{1,2}$. We also know that there are at most one vertex in $V _{1,1}$. Hence when $k\ge 2$ there are one or more  vertices in either $V_{1,2}$ or $V_{1,3}$. \newline

Without loss of generality, assume there are more vertices in set $V_2$. And for graph $T_{n,2}$ we also assume there are more vertices in $V_{2,2}$ We set the test vector to be 
$$
\bold{x}(i) = 
\begin{dcases}
0\quad &i=1,n+1\\
1 \quad &i\in V_{1,2} \text{ or } i\in V_{2,2}\\
-1\quad & \text{otherwise}
\end{dcases}.
$$
We notice that the test vector here is basically the same when we define the test vector for the graph $T_n^m$. For elements of $T_{n,1}$, the number of occurrences of $1$ is $\frac{n-1}{2}$ and the number of occurrences of $-1$ is $\frac{n-1}{2}$. For elements of $T_{n,1}$ the number of occurrences of $1$ is $\frac{n-1}{2}$ and the number of occurrences of $-1$ is $\frac{n-1}{2}$. Hence we have 
\begin{align*}
(\bold{x},\bold{1})&=\sum_{i=1}^{2n} x(i)\\
&=x(1)+x(n+2) + \sum_{i\in V_{1,2}}x(i)+\sum_{i\in V_{2,2}}^ x(i)+\sum_{i\in V_{1,3}}x(i)+\sum_{i\in V_{2,3}}x(i)\\
&=0+0+\frac{n-1}{2}+\frac{n-1}{2}-\frac{n-1}{2}-\frac{n-1}{2}\\
&=0.
\end{align*}
We have finished verifying $(\bold{x},\bold{1})=0$. Now we can try to bound $\lambda_2(T_n^m)$:
\begin{align*}
\lambda_2(T_n^{2\times k})&\le \frac{\bold{x}^T\bold{L}\bold{x}}{\bold{x}^T\bold{x}}\\   
&=\frac{\sum_{(a,b)\in E_{T_n^{2\times k}}} (x(a)-x(b))^2}{\sum_{i=1}^{2n}x(i)^2}\\
& = \frac{\sum_{1\le i,j\le n} (x(i)-x(j))^2}{\sum_{1\le i\le n} x(i)^2+\sum_{n+m-1\le i\le 2n} x(i)^2}\\
&+ \frac{\sum_{(i,j)\in {\{e_1,e_2,\dots e_k\}}} (x(i)-x(j))^2}{\sum_{1\le i\le n} x(i)^2+\sum_{n+m-1\le i\le 2n}x(i)^2}\\
&+ \frac{\sum_{n+m-1\le i,j\le 2n} (x(i)-x(j))^2}{\sum_{1\le i\le n} x(i)^2+\sum_{n+m-1\le i\le 2n}x(i)^2}.
\end{align*}
Notice that the first term is 
\begin{align*}
&\frac{\sum_{1\le i,j\le n} (x(i)-x(j))^2}{\sum_{1\le i\le n} x(i)^2+\sum_{n+m-1\le i\le 2n} x(i)^2}\\
=&\frac{ (x(1)-x(2))^2+(x(2)-x(3))^2}{\sum_{1\le i\le n} x(i)^2+\sum_{n+1\le i\le n+1} x(i)^2+\sum_{n+1\le i\le 2n}x(i)^2}\\
=&\frac{2}{n-1+n-1}\\
=&\frac{1}{n-1}.
\end{align*}
The second term is
\begin{align*}
\frac{\sum_{(i,j)\in {\{e_1,e_2,\dots e_k\}}} (x(i)-x(j))^2}{\sum_{1\le i\le n} x(i)^2+\sum_{n+m-1\le i\le 2n}x(i)^2}
&=\frac{\sum_{i\in V_{1,1} \text{ or } j\in V_{2,1}}(x(i)-x(j))^2}{\sum_{1\le i\le n} x(i)^2+\sum_{n+m-1\le i\le 2n}x(i)^2}\\
&+\frac{\sum_{i\in V_{1,2}, j\in V_{2,2}(x(i)-x(j))^2}}{\sum_{1\le i\le n} x(i)^2+\sum_{n+m-1\le i\le 2n}x(i)^2}\\
&+\frac{\sum_{i\in V_{1,2}, j\in V_{2,3}(x(i)-x(j))^2}}{\sum_{1\le i\le n} x(i)^2+\sum_{n+m-1\le i\le 2n}x(i)^2}\\
&+\frac{\sum_{i\in V_{1,3}, j\in V_{2,2}(x(i)-x(j))^2}}{\sum_{1\le i\le n} x(i)^2+\sum_{n+m-1\le i\le 2n}x(i)^2}\\
&+\frac{\sum_{i\in V_{1,3}, j\in V_{2,3}(x(i)-x(j))^2}}{\sum_{1\le i\le n} x(i)^2+\sum_{n+m-1\le i\le 2n}x(i)^2}\\
\end{align*}
We know that there are at most $\frac{n-1}{2}$ vertices which have nonzero value connected to vertex $1$ and at most $\frac{n-1}{2}$ vertices which have nonzero value connected to vertex $n+1$. Hence we have
$$
\frac{\sum_{i\in V_{1,1} \text{ or } j\in V_{2,1}}(x(i)-x(j))^2}{\sum_{1\le i\le n} x(i)^2+\sum_{n+m-1\le i\le 2n}x(i)^2} \le \frac{2(n-1)(0-1)^2}{2(n-1}=1.
$$
The second term in the sum above is $0.$ We know that $i=j\le \frac{n-1}{2}$, where $i\in V_{1,2},j\in V_{2,3}$. Thus, the third term in the sum above obeys
$$
\frac{\sum_{i\in V_{1,2}, j\in V_{2,3}}(x(i)-x(j))^2}{\sum_{1\le i\le n} x(i)^2+\sum_{n+m-1\le i\le 2n}x(i)^2} < \frac{(n-1)(1+1)^2}{2(n-1)}=2.
$$
We know that  $i=j\le \frac{n-1}{2}$ where $i\in V_{1,3},j\in V_{2,2}$, Thus, the fourth term in the sum above satisfies 
$$
\frac{\sum_{i\in V_{1,3}, j\in V_{2,2}}(x(i)-x(j))^2}{\sum_{1\le i\le n} x(i)^2+\sum_{n+m-1\le i\le 2n}x(i)^2} < \frac{(n-1)(-1-1)^2}{2(n-1)}=2.
$$
Lastly, the fifth term in the sum above is $0$.\\

We can also see that 
\begin{align*}
&\frac{\sum_{n+m-1\le i,j\le 2n} (x(i)-x(j))^2}{\sum_{1\le i\le n} x(i)^2+\sum_{n+m-1\le i\le 2n}x(i)^2}\\
=&\frac{ (x(n+1)-x(n+2))^2+(x(n+1)-x(n+3))^2}{\sum_{1\le i\le n} x(i)^2+\sum_{n+1\le i\le 2n}x(i)^2}\\
=&\frac{2}{n-1+n-1}\\
=&\frac{1}{n-1}.
\end{align*}
But we don't need to add above terms when $k$ is relatively small. We can make the inequality tighter depending on the value of $k$. We notice that when $k\le n-1$ the $\lambda_2({T_n^{2\times k}})$ takes the greatest value when $i\in V_{1,3},j\in V_{2,2}$ or $i\in V_{1,2}, j\in V_{2,3}$. Hence, we actually have 
$$
\lambda_2({T_n^{2\times k}})\le \frac{m(-1-1)^2}{2(n-1)}+\frac{2}{n-1}=\frac{2m+2}{n-1}
$$
when $k> n-1$. Also, $\lambda_{T_n^{2\times k}}$ takes the greatest value when $i\in V_{1,3},j\in V_{2,2}$ or $i\in V_{1,2}, j\in V_{2,3}$. It's follows that
\begin{align*}
\lambda_{T_2(n^{2\times k}})&\le 2+ \frac{(m-(n-1))(0+-1)^2}{2(n-1)}+\frac{2}{n-1} \\
&=\frac{m+3n+1}{2(n-1)} \\
&=\frac{2m+2}{n-1}+\frac{-3m+3n-3}{2(n-1)}.
\end{align*}

From the above we get that 
\begin{align*}
\lambda_{T_n^{2\times k}}\le \frac{2m+2}{n-1}+\chi_{k\geq n-1}\frac{-3m+3n-3}{2(n-1)}
\end{align*}
\end{proof}
\end{theorem}
Now we can construct a graph $T_{n\times l}^2$ which is connected by $l$ identical full binary graphs $T_{n,1},\dots T_{n,l}$ using single edge. For every graph $T_{n,j}$ where $1\le j \le l-1$, there is a vertex $v^i$ which is ancestor of all other vertices in $T_{n,j}$, we connect that with $v^{j+1}$. As an example, we display $T_{7\times 3}^2$:
\begin{center}
\includegraphics[scale = 0.7]{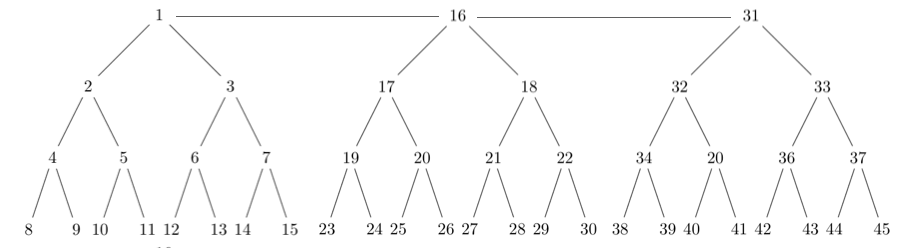}
\end{center}
We have following theorem.

\begin{theorem}
For the graphs $T_{n\times l}^2$ which we described above we have following bound of the second eigenvalues:
$$
\frac{2}{(nl-1)(l\cdot \log_2(n-1)-1)}\le T_{n\times l}^2 \le \frac{l}{n-1}.
$$
\begin{proof}
For the vertex $v^i$ which is ancestor of all other vertices in $T_{n,i}$, we label it as $(j-1)n+1$. Then $(j-1)n+1$ has two children. We label them as $(j-1)n+2$ and $(j-1)n+3$. Then we label children of $(j-1)n+2$ as $(j-1)n+4$ and $(j-1)n+5$, the children of $(j-1)n+3$ as $(j-1)n+6$ and $(j-1)n+7$, and so on until $jn$. 
To get the upper bound we still need to use a test vector. We can set the test vector to be:
$$
\bold{x}(i) = 
\begin{dcases}
0\quad &j\in\{1,\dots l\},(j-1)n+1\\
1 \quad &j\in\{1,\dots l\},(j-1)n+2\\
1 \quad &j\in\{1,\dots l\},(j-1)n+2<i\le jn \text{ and } i  \text{ descendant of } (j-1)n+2\\
-1\quad & \text{otherwise}
\end{dcases}.
$$
We notice that for elements of $T_{n,j}$ where $j\in \{1,\dots l\}$, the number of occurrences of $1$'s in $\bold{x}$ is $\frac{n-1}{2}$ and the number of occurrences of $-1$'s in $\bold{x}$ is $\frac{n-1}{2}$. Hence we have 
\begin{align*}
(\bold{x},\bold{1})&=\sum_{j=1}^{l}\sum_{i=(j-1)n+1}^{jn} x(i)\\
&=l(0)+l\frac{n-1}{2}-l\frac{n-1}{2}\\
&=0.
\end{align*}
We have finished verifying $(\bold{x},\bold{1})=0$. Now we estimate the upper bound of $\lambda_2(T_n^m)$ :
\begin{align*}
\lambda_2(T_{n\times l}^2)&\le \frac{\bold{x}^T\bold{L}\bold{x}}{\bold{x}^T\bold{x}}\\   
&=\frac{\sum_{(a,b)\in E_{T_{n\times l}^{2}}} (x(a)-x(b))^2}{\sum_{i=1}^{nl}x(i)^2}.\\
\end{align*}
We notice that only when $a=(j-1)n+1$ with $b=(j-1)n+2$ or $b=(j-1)n+3$, then the term $(x(a)-x(b))^2$ is not zero. There are $n-1$ vertices in each $T_{n,j}$ such that $x{a}^2=1$. Hence the above equation has the following form:
\begin{align*}
& = \frac{\sum_{j=1}^{l}\sum_{1\le a,b\le nj} (x(a)-x(b))^2}{\sum_{j=1}^{l}\sum_{1\le a\le nj} x(a)^2}\\
& = \frac{2l}{(n-1)l}\\
& = \frac{2}{(n-1)}.
\end{align*} \newline
Now we need to find the lower bound. We still compare our graphs to complete graphs. For every pair of edge $(a,b)\in E_{T_n^m}$, let the path graph $P_{a,b}$ be a path from $a$ to $b$, and $G_{a,b}$ be a graph with a single edge $(a,b)$; then from the lemma, we have $|P_{a,b}|P_{a,b} \succcurlyeq G_{a,b}$. We notice that length of the path $P_{a,b}$ is the longest when $a$ and $b$ where $a\in T_{n,1}$ and $b\in T_{n,l}$, and $a$ and $b$ have no descendants. Hence the length of the path $P_{a,b}$ is at most $l \cdot \log_2{(n-1)}-1$, which means that $|P_{a,b}|\le 2 \log_l(n-1)-1$. Now,
\begin{align*}
G_{a,b}&\preccurlyeq |P_{a,b}|P_{a,b} \\ 
&\preccurlyeq(l \cdot \log_2(n-1)-1)P_{a,b} \\
&\preccurlyeq(l \cdot \log_2(n-1)-1)T_{n\times l}^m.
\end{align*}
Also, we know that complete graph $K_{nl}$ has $\binom{nl}{2}$  edges, so $$K_{nl}\preccurlyeq \sum_{(a,b)\in E_{K_{nl}}}G_{a,b}\preccurlyeq \binom{nl}{2} G_{a,b} \preccurlyeq \binom{nl}{2}(llog_2(n-1)-1) T_m^n.
$$
Hence,
$$
nl=\lambda_2(K_{nl}) \le \binom{nl}{2}(l\cdot \log_2(n-1)-1) \lambda(T_{n\times l}^2).
$$
From the above, we get that $\lambda_2(T_m^n)\geq \frac{2}{(nl-1)(l \log_2(n-1)-1)}$.

\end{proof}
\end{theorem}
We have finished discussing the graph $T_{n\times l}^2$. Now we will discuss a more general case when a graph $B_{n\times l}^2$ which is connected by $l$ identical graphs $G_{n,1},\dots G_{n,l}$ using a single edge. Graphs $G_{n,1},\dots G_{n,l}$ are all identical, but they can be any arbitrary graph. We have following theorem.
\begin{theorem}
For graphs of the form $B_{n\times l}^2$, which we described above, we have following bound on the second eigenvalue:
$$
0<\lambda_2(B_{n\times l}^2)\le 2
$$
\end{theorem}
\begin{proof}
Now we need to label graph $B_{n\times l}^2$ first. For graph $G_{n,1}$, we know that there is a vertex which is attached to another vertex in graph $G_{n,2}$; we label this vertex as $n$. All other vertices in the graph $G_{n,1}$ can be labeled from $1$ to $n-1$ without repeating. For graph $G_{n,l}$, we know that there is a vertex which is attached to another vertex in graph $G_{n,l-1}$; we label this vertex as $n(l-1)+1$. All other vertices in the graph $G_{n,l}$ can be labeled from $n(l-1)+2$ to $nl$ without repeating. For every graph $G_{n,i}$ where $2\le i \le l-1$, we know that there is a vertex which is attached to another vertex in graph $G_{n,i-1}$. We label this vertex as $n(i-1)+1$, and there is a vertex which is attached to another vertex in graph $G_{n,i+1}$. We label that vertex as $ni$. All other vertices in the graph $G_{n,i}$ can be labeled from $n(i-1)+2$ to $ni-1$ without repeating.\\

We notice that $B_{n\times l}^2$ is connected. Hence we have 
$$
\lambda_2(B_{n\times l}^2)>0.
$$
Also, we notice that if we take away an edge $(n,n+1)$ from graph $B_{n\times l}^2,$ then we get a new graph $\tilde{B}_{n\times l}^2=B_{n\times l}^2\setminus(n,n-1),$ which is not connected. Thus, we have $\lambda_2({\tilde{B}_{n\times l}^2})=0$\\
From Theorem 3.2, we know that 
$$
0=\lambda_2(B_{n\times l}^2)=\min\limits_{\substack{(\bold{x},\bold{1})=0,x\in R }}{\frac{\bold{x}^T \bold{L_{B_{n\times l}^2}}\bold{x}}{\bold{x}^T\bold{x}}}
$$
and
$$
\lambda_2({\tilde{B}_{n\times l}^2})=\min\limits_{\substack{(\bold{x},\bold{1})=0,x\in R }}{\frac{\bold{x}^T \bold{L_{\tilde{B}_{n\times l}^2}}\bold{x}}{\bold{x}^T\bold{x}}}.
$$
Notice that 
\begin{align*}
\lambda_2(B_{n\times l}^2)&=\min\limits_{\substack{(\bold{x},\bold{1})=0,x\in R }}{\frac{\bold{x}^T \bold{L_{B_{n\times l}^2}}\bold{x}}{\bold{x}^T\bold{x}}}\\
&=\min\limits_{\substack{(\bold{x},\bold{1})=0,x\in R }}\, \frac{1}{\bold{x}^T\bold{x}}\sum\limits_{(a,b)\in B_{n\times l}^2}(x(a)-x(b))^2\\
&\le \min\limits_{\substack{(\bold{x},\bold{1})=0, x\in R }}\, \left[\sum\limits_{\substack{(a,b)\in B_{n\times l}^2, \\ (a,b)\neq (n,n+1)}}\frac{(x(a)-x(b))^2}{{\bold{x}^T\bold{x}}}+\frac{(x(n)-x(n+1))^2}{{\bold{x}^T\bold{x}}}\right]\\ 
&= \min\limits_{\substack{(\bold{x},\bold{1})=0,x\in R }}\left[\sum\limits_{(a,b)\in \tilde{B}_{n\times l}^2}\frac{(x(a)-x(b))^2}{{\bold{x}^T\bold{x}}}+\frac{(x(n)-x(n+1))^2}{{\bold{x}^T\bold{x}}}\right]\\
&= \min\limits_{\substack{(\bold{x},\bold{1})=0,x\in R }}\left[\sum\limits_{(a,b)\in \tilde{B}_{n\times l}^2}\frac{(x(a)-x(b))^2}{{\bold{x}^T\bold{x}}} +\frac{(x(n)-x(n+1))^2}{{\bold{x}^T\bold{x}}}\right] \\
&=  \min\limits_{\substack{(\bold{x},\bold{1})=0,x\in R }} (I_1 + I_2)
\end{align*}
where 
$$I_1 = \sum\limits_{(a,b)\in \tilde{G}_{n\times l}^2}\frac{(x(a)-x(b))^2}{{\bold{x}^T\bold{x}}} \text{ and } I_2 = \frac{(x(n)-x(n+1))^2}{{\bold{x}^T\bold{x}}}.$$
Notice that $I_1$ is the same as 
$$
0=\lambda_2(\tilde{B}_{n\times l}^2)=\min\limits_{\substack{(\bold{x},\bold{1})=0,x\in R }}{\frac{\bold{x}^T \bold{L_{\tilde{B}_{n\times l}^2}}\bold{x}}{\bold{x}^T\bold{x}}}=  \frac{1}{\bold{x}^T\bold{x}}\sum\limits_{(a,b)\in \tilde{B}_{n\times l}^2}(x(a)-x(b))^2.
$$
We have the following bound for the denominator of $I_2$:
\begin{align*}
\bold{x}^T\bold{x}&=\sum_{a=1}^{a=nl}(x(a))^2\\
&\geq (x(n))^2+(x(n+1))^2\\
&= \frac{1}{2}(x(n))^2+\frac{1}{2}(x(n+1))^2+\frac{1}{2}(x(n))^2+\frac{1}{2}(x(n+1))^2-x(n)x(n+1)+x(n)x(n+1)\\
&= \frac{1}{2}((x(n))^2+(x(n+1))^2+2x(n)x(n+1))+\frac{1}{2}((x(n))^2+(x(n+1))^2-2x(n)x(n+1))\\
&= \frac{1}{2}(x(n)+x(n+1))^2+\frac{1}{2}(x(n)-x(n+1))^2\\
& \geq \frac{1}{2}(x(n)-x(n+1))^2.
\end{align*}
We now use this to bound the second term.
$$
{\frac{(x(n)-x(n+1))^2}{\bold{x}^T\bold{x}}}\le \frac{(x(n)-x(n+1))^2}{\frac{1}{2}(x(n)-x(n+1))^2}=2.
$$
Adding two terms together gets us
$$
\lambda_2(B_{n\times l}^2)\le 2.
$$
Thus, the result is proven.
\end{proof}
\begin{remark}
The upper bound in the inequality above cannot be improved. We notice that when $G_{1\times 2}^2=P_2,$ we know that the graph $P_2$ is constructed by connecting two identical graphs $G_{1,1}$ together, where $G_{1,1}$ is a single vertex. Then we have $\lambda_2(G_{1\times 2}^2)=2$, so equality is achieved.
\end{remark}
\end{section}
\begin{section}{Conclusion and future work}
Our work from section 4 to section 6 went through various different graphs. We noticed that for the $K_n$ type of graphs $D_n^m$, we have the approximate bound  $\lambda_2(D_n^m)\sim\frac{1}{n}$. We also observe that when $n$ and $m$ increase the second eigenvalues decrease. Similarly, for $D_n^{2\times k}$, we have $\lambda_2(D_n^{2\times k})\sim\frac{1}{n}$, and when $n$ increases the second eigenvalues decrease too. \newline

We noticed that for $S_n$ type graphs, $S_n^m$, we still have $\lambda_2(S_n^m)\sim\frac{1}{n}$. This is expected because star graph is a type of complete bipartite graph $K_{1,n-1}$, and complete graphs can also be complete bipartite graphs depending on the choice of $n$. But the $T_n$ type of bridge graph is different from the first two types of graphs. We noticed the lower bound of the second eigenvalues of  $T_n^m$ and $T_n^{2\times k}$ and $T_{n\times}^{l}$ are all dependent on $\log(n+1)$. Also, the upper bound is still asymptotically dependent on $\frac{1}{n}$. \newline

From the above proofs, we noticed that the test vector method is a very good technique for upper bound of the eigenvalues. This is because theorem 3.2 enables us to find a test vector which is orthogonal to the first eigenvector, and from theorem 3.3 we know that the first eigenvector is $\bold{1}$. It's also important to use theorem 6.4, which bounded general bridge graphs $B_{n\times l}^2$. Now we are curious what will happen if we construct a bridge graph like $B_{n\times \infty}^2$. It will not be the same as the case when $l$ is finite because we can't count the vertices one by one anymore. However, we still want to know if the results are somewhat similar.\newline

Our future work will be constructing infinite bridge graphs. Now we can start from some basic definition and related theorems.
\begin{definition}
An infinite graph $G=(V,E)$ is a graph which has a countably infinite number of vertices. Infinite Bridge graphs $G_{n\times \infty}^m$ are constructed by using path graphs, $P_m$ with $n\geq 2$, and gluing together a countably infinite number of identical finite graphs on each end of the paths. Usually we can find an invertible map from vertex set $V$ to $\mathbb{Q}$. We only discuss unweighted graphs.
\end{definition}
For infinite graphs we cannot use adjacency matrices anymore. Now we are seeking a substitution of matrices to associate our graph with an operator. Definitions of operators related to infinite graphs are mentioned by other authors like  Bojann Mohar\cite{b4} and Dragos M. Cvetokvic\cite{b3}, and Ayadi Hela\cite{b5} also defined the Laplacian Operator. But since we are only interested in infinite bridge graphs, we will use a different definition than other authors. The following is an example formed by attaching a countably infinite number of $K_8$ graphs together:
\begin{center}
\includegraphics[scale = 0.5]{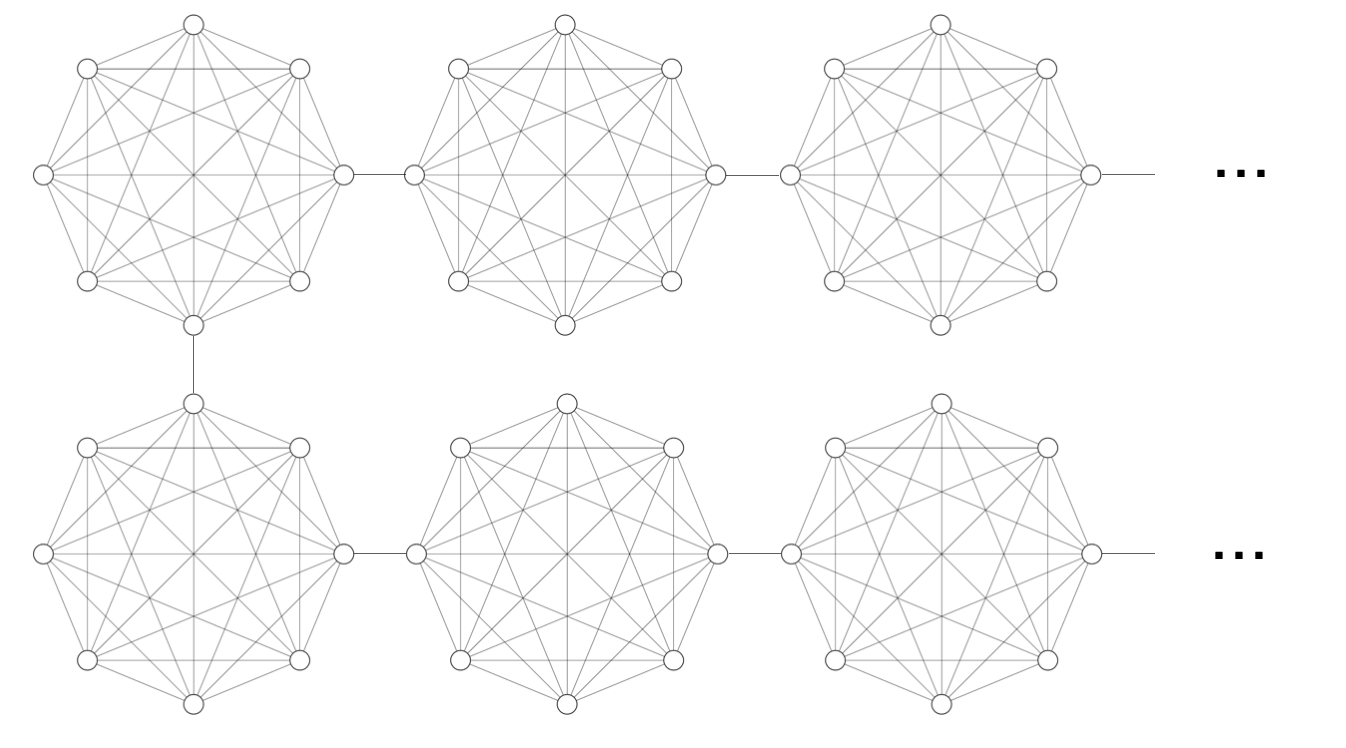}
\end{center}

\begin{definition} The space $\ell^2(\mathbb{N} \times \mathbb{N})$ is defined as the space of sequences $\{x_{i,j}\}_{i.j \in \mathbb{N}}$ such that 
$$\sum_{\mathbb{N} \times \mathbb{N}} |x_{i,j}|^2 < \infty.$$
\end{definition} 

\begin{definition} The adjacency operator $M$ of a weighted graph $G=(V,E)$ is defined as operator with the following entries
$$
M(a,b) = 
\begin{dcases}
1 \quad (a,b)\in E\\
0 \quad (a,b)\notin E.
\end{dcases}
$$ Notice that $\{M(a,b)\}_{a,b \in \mathbb{N}}$ forms a sequence.
\end{definition}
\begin{definition}

The degree operator $D$ of a graph $G=(V,E)$ is a diagonal matrix whose entries are given by
$$
D(a,b) = 
\begin{dcases}
d(a)\quad &a=b\\
0 \quad &a\neq b.
\end{dcases}
$$ Like before, $\{D(a,b)\}_{a,b \in \mathbb{N}}$ forms a sequence.
\end{definition}

\begin{definition}
The graph laplacian operator $L$ of a graph $G$ is defined to be
$$
L =M - D,
$$ where the subtraction operation is subtracting corresponding elements in each sequence.
\end{definition}
\begin{theorem}
The adjacency operator, degree operator and laplacian operator are all linear operators.
\begin{proof}
The proof is straightforward from the definition of the operators.
\end{proof}
\end{theorem}
\begin{theorem}
For the laplacian operator  $L_G: \mathcal{X} \rightarrow \mathcal{Y}$ with $\mathcal{X}, \mathcal{Y} \subset \ell^2(\mathbb{N})$, where  $G$ is a bridge graph, $L$ is a well defined mapping.
\begin{proof}
We notice that for $a\in V_G$, we have 
$$Lx(a)=\sum_{b\in N(a)}(x(a)-x(b)).$$
Since we assume that $x\in \ell^2(\mathbb{N})$, we know there is an $M$ such that $(\sum_{a=1}^{\infty}  \|a\|_2)^{\frac{1}{2}}<M$ for some $M > 0$. Since $G$ is a bridge graph, then we know that for every vertex $a$, then $a$ has a finite number of vertices in its neighborhood. Hence we have 
\begin{align*}
\| y\|_2 &=\| L x \|_2\\
&=\left(\sum_{a=1}^{\infty}\sum_{b\in N(a)}(x(a)-x(b))^2\right)^{\frac{1}{2}}\\
&\le \left(\sum_{a=1}^{\infty}m \cdot \max_{b\in N(a)}\left\{\|a\|_2^2,\|b\|_2^2\right\}\right)^{\frac{1}{2}}\\
&\le \left(\sum_{a=1}^{\infty} m^2 \|a\|_2^2\right)^{\frac{1}{2}}\\
&=m\| L  x \|_2 \\
&< mM.
\end{align*}
Hence from the second last line of the above equations we get that Laplacian operator $L_G$ is a bounded operator and $y\in \ell^2(\mathbb{N})$.
\end{proof}
\end{theorem}
The following definitions are from Elias M. Stein and Rami Shakarchi\cite{b2}.
\begin{definition}
It is well known that $\ell^2(\mathbb{N})$ is a Hilbert space. Therefore, it admits an inner product. For vector $x,y\in \ell^2(\mathbb{N})$, the inner product of $x$ and $y$ is defined as 
$$(x,y)= \sum_{n \in \mathbb{N}} x_n y_n.$$
\end{definition}
\begin{definition}
For a graph operator $T$, if we have
$$
T\phi=\lambda \phi,
$$
then we call $\lambda$ the eigenvalue and $\phi$ the eigenvector corresponding to eigenvalue $\lambda$ for operator $T$.
\end{definition}
\begin{definition}
We say $\lambda\in \mathbb{R}$ is in the spectrum of $A$ if $A-\lambda I$ has no bounded inverse. The spectrum is denoted by $\sigma({A})$ where $\sigma({A}) \subset \mathbb{R}$, and the resolvent set for $A$ is $\rho({A})=\mathbb{R}\setminus\sigma({S})$.
\end{definition}
Our future work will be about the spectrum of Laplacian operator.
\end{section}
\newpage


\begin{thebibliography}{00}

\bibitem{b1} Dan Spielman. Spectral and Algebraic Graph Theory, 2019.
\bibitem{b2} Elias M. Stein and Rami Shakarchi. Functional Analysis: Introduction to Further Topics in Analysis. Princeton University Press, 2011.
\bibitem{b3} Dragos M. Cvetokvic, Michael Doob, Ivan Gutman, Aleksandar Toragasev. Recent Results in the Theory of Graph Spectra. Elsevier Science Publishers B.V., 1988.
\bibitem{b4} Bojan Mohar and Wolfgang Woess. A Survey on Spectra of Infinite Graphs. Bull. London Math. Soc. 21 (1989) 209-234, 1989.
\bibitem{b5} Ayadi Hela, Hèla Ayadi. Spectra of Laplacians on Forms on an Infinite Graph. Operators and Matrices, 2017. ffhal-01710216
\end{thebibliography}
\end{document}